\documentclass{amsart}

\usepackage{amssymb,amsmath,bm,enumerate,stmaryrd,mathtools}

\usepackage[shortlabels]{enumitem}
\usepackage{colonequals}
\usepackage{eucal}
\usepackage{bbm}
\usepackage{multicol}

\usepackage[colorlinks]{hyperref}

\usepackage[capitalise]{cleveref}
\crefformat{equation}{(#2#1#3)}
\crefrangeformat{equation}{(#3#1#4--#5#2#6)}
\crefmultiformat{equation}{(#2#1#3}{ \& #2#1#3)}{, #2#1#3}{ \& #2#1#3)}

\crefformat{enumi}{(#2#1#3)}
\crefrangeformat{enumi}{(#3#1#4--#5#2#6)}
\crefmultiformat{enumi}{(#2#1#3}{ \& #2#1#3)}{, #2#1#3}{ \& #2#1#3)}

\usepackage{tikz-cd}

\tikzcdset{
 column sep/wide/.initial=+6em,
 column sep/Wide/.initial=+7.2em,
 column sep/wider/.initial=+8.4em,
 column sep/Wider/.initial=+9.6em,
 column sep/broad/.initial=+10.8em,
 column sep/Broad/.initial=+12em,
 column sep/broader/.initial=+13.2em,
 column sep/Broader/.initial=+14.4em,
}

\newlength{\perspective}

\newcommand{\nospacepunct}[1]{\makebox[0pt][l]{\,#1}}

\newtheorem{theorem}[subsection]{Theorem}
\newtheorem{proposition}[subsection]{Proposition}
\newtheorem{lemma}[subsection]{Lemma}
\newtheorem{corollary}[subsection]{Corollary}

\theoremstyle{definition}
\newtheorem{definition}[subsection]{Definition}
\newtheorem{example}[subsection]{Example}
\newtheorem{discussion}[subsection]{}

\theoremstyle{remark}
\newtheorem{remark}[subsection]{Remark}

\newtheorem*{ack}{Acknowledgements}

\numberwithin{equation}{subsection}

\theoremstyle{plain}
\newcounter{intro}
\newtheorem{introthm}[intro]{Theorem}

\crefname{introthm}{Theorem}{Theorems}
\crefname{example}{Example}{Examples}

\renewcommand{\epsilon}{\varepsilon}
\renewcommand{\phi}{\varphi}
\renewcommand{\theta}{\vartheta}

\renewcommand{\mod}[1]{\operatorname{mod}(#1)}
\newcommand{\Mod}[1]{\operatorname{Mod}(#1)}
\newcommand{\Bimod}[2]{\operatorname{Bimod}(#1,#2)}

\newcommand{\Stmod}[1]{\underline{\operatorname{Mod}}(#1)}
\newcommand{\Ch}[1]{\operatorname{Ch}(#1)}

\newcommand{\dcat}[2][]{{\rm D}_{#1}(#2)}
\newcommand{\dbcat}[1]{\dcat[b]{#1}}
\newcommand{\Perf}[1]{\operatorname{Perf}(#1)}

\newcommand{\BZ}{\mathbb{Z}}

\newcommand{\bmalpha}{\bm{\alpha}}

\newcommand{\cat}[1]{{\mathcal{#1}}}

\DeclareMathOperator{\cone}{cone}
\DeclareMathOperator{\cyl}{cyl}
\DeclareMathOperator{\ctr}{Z}

\newcommand{\End}[3][*]{\operatorname{End}^{#1}_{#2}(#3)}
\newcommand{\env}[2][]{{#2}^e_{#1}}
\newcommand{\Ext}[4][*]{\operatorname{Ext}^{#1}_{#2}(#3,#4)}

\newcommand{\HH}[2][*]{\operatorname{HH}^{#1}(#2)}
\DeclareMathOperator{\Ho}{Ho}
\renewcommand{\hom}[3][]{\operatorname{hom}_{#1}(#2,#3)}
\newcommand{\homl}[3][]{\operatorname{hom}^\ell_{#1}(#2,#3)}
\newcommand{\homr}[3][]{\operatorname{hom}^r_{#1}(#2,#3)}

\DeclareMathOperator{\id}{id}
\newcommand{\Kos}{\operatorname{Kos}}
\newcommand{\kosobj}[2]{#1/\!\!/#2}
\newcommand{\level}[3]{\operatorname{level}_{#1}^{#2}(#3)}

\newcommand{\lotimes}{\otimes^{\mathbf{L}}}
\newcommand{\op}[1]{#1^{\rm{op}}}
\DeclareMathOperator{\projdim}{projdim}
\newcommand{\RHom}[4][]{\operatorname{RHom}^{#1}_{#2}(#3,#4)}
\newcommand{\Rhoml}[3][]{\operatorname{Rhom}^\ell_{#1}(#2,#3)}
\newcommand{\Rhomr}[3][]{\operatorname{Rhom}^r_{#1}(#2,#3)}
\newcommand{\set}[2]{\left\{#1 \,\middle|\, #2\right\}}

\newcommand{\susp}{\Sigma}
\newcommand{\thick}[3][]{\operatorname{thick}^{#1}_{#2}(#3)}
\newcommand{\unit}{\mathbbm{1}}

\title[Generation time for biexact functors and Koszul objects]{Generation time for biexact functors and Koszul objects in triangulated categories}
\date{\today}

\keywords{Biexact functor, triangulated category, Verdier structure, Koszul object, cofibration category}

\subjclass[2020]{18G80 (primary), 55U35}

\author[J.~C.~Letz]{Janina C. Letz}
\address{Janina~C.~Letz,
Faculty of Mathematics,
Bielefeld University,
PO Box 100 131,
33501 Bielefeld,
Germany}
\email{jletz@math.uni-bielefeld.de}

\author[M.~Stephan]{Marc Stephan}
\address{Marc Stephan,
Faculty of Mathematics,
Bielefeld University,
PO Box 100 131,
33501 Bielefeld,
Germany}
\email{marc.stephan@math.uni-bielefeld.de}

\begin{document}

\begin{abstract}
This paper concerns the generation time that measures the number of cones necessary to obtain an object in a triangulated category from another object. This invariant is called level. We establish level inequalities for enhanced triangulated categories: One inequality concerns biexact functors of topological triangulated categories, another Koszul objects. In particular, this extends inequalities for the derived tensor product from commutative algebra to enhanced tensor triangulated categories. We include many examples.
\end{abstract}

\maketitle


\section{Introduction}

Triangulated categories appear in many areas, such as algebraic topology, representation theory or commutative algebra. One approach to understand a triangulated category is to study how objects generate each other. An object $X$ generates an object $Y$, if $Y$ can be obtained from $X$ by taking cones, suspensions and retracts. 

Following \cite{Avramov/Buchweitz/Iyengar/Miller:2010} we call the generation time \emph{level}. Loosely speaking the $X$-level measures the number of cones required to build objects from $X$. It generalizes some well-known invariants: When $\cat{T} = \dcat{\Mod{R}}$ the derived category of modules over a ring $R$, then the $R$-level of a module coincides with its projective dimension $+1$; see for example \cite{Christensen:1998}. For a local ring $R$ with residue field $k$, the $k$-level of a module coincides with its Loewy length. The $k$-level of a perfect complex $F$ has been used in \cite{Avramov/Buchweitz/Iyengar/Miller:2010} to establish a rank inequality for the homology of $F$.

Further, level is closely connected to the Rouquier dimension of a triangulated category introduced in \cite{Rouquier:2008}. For triangulated categories with finite Rouquier dimension, there are some Brown representability theorems; see \cite{Bondal/VanDenBergh:2003,Rouquier:2008,Letz:2023}. For many examples of triangulated categories it is known whether the Rouquier dimension is finite or infinite, but in the former case the exact value is rarely known. We expect that studying level will help computing the Rouquier dimension.

Besides the properties following from the definition \cite[Lemma~2.4]{Avramov/Buchweitz/Iyengar/Miller:2010}, not much is known about the behavior of level. Many estimates for level are rather rough. For example the transitivity of finite building yields a product inequality for level. In this paper we give a refinement for biexact functors. The following inequality is optimal in that equality can be achieved.

\begin{introthm} \label{intro:biexact_level}
Let $\mathsf{F} \colon \cat{S} \times \cat{T} \to \cat{U}$ be a biexact functor of topological triangulated categories. Then 
\begin{equation*}
\level{\cat{U}}{\mathsf{F}(X,X')}{\mathsf{F}(Y,Y')} \leq \level{\cat{S}}{X}{Y} + \level{\cat{T}}{X'}{Y'} - 1
\end{equation*}
for $X,Y \in \cat{S}$ and $X',Y' \in \cat{T}$.
\end{introthm}

For this result see \cref{biexact_level} and \cref{cofib_biexact_verdier}. \cref{intro:biexact_level} extends and generalizes \cite[Lemma~2.4]{Avramov/Iyengar/Neeman:2018} where the inequality is shown for $\cat{S} = \cat{T} = \cat{U} = \dcat{\Mod{R}}$ the derived category of a commutative ring $R$, $\mathsf{F}$ the derived tensor product and $X = X' = R$. In \cref{sec:examples} we provide many examples of biexact functors and triangulated categories for which the above inequality holds. These include the tensor product on the derived category of a commutative ring or a group ring, on the stable module category of a cocommutative Hopf algebra, and on various categories of spectra.

In the proof of \cref{intro:biexact_level} we construct the triangles building $\mathsf{F}(Y,Y')$ by taking homotopy pushouts. In a general triangulated category one does not have enough control over the homotopy pushout and its compatibilities, since it is neither a honest pushout nor captures the full homotopical information. For this reason we work with biexact functors that admit a strong Verdier structure; see \cite[Theorem~3.5]{Keller/Neeman:2002} and \cite[Definition~3.31]{Aldrovandi/Lester:2023}. This extra structure ensures that the homotopy pushout in $3 \times 3$ diagrams induced by applying the biexact functor to exact triangles in each component has the desired compatibilities. 

In practice, many biexact functors of triangulated categories admit a strong Verdier structure. We show that whenever a triangulated category has an enhancement and the biexact functor respects the enhancement, the biexact functor admits a strong Verdier structure. For monoidal products this was shown by \cite{May:2001,Groth/Ponto/Shulman:2014,Groth/Stovicek:2018}. We consider general bifunctors. Explicitly, we prove that for any biexact functor of stable cofibration categories the induced bifunctor on the homotopy categories is biexact and admits a strong Verdier structure; see \cref{cofib_biexact_verdier}. Stable cofibration categories provide a convenient setting, since their homotopy categories are precisely the topological triangulated categories \cite{Schwede:2013}, and the homotopy theory of cofibration categories is equivalent to the homotopy theory of finitely cocomplete $\infty$-categories \cite{Szumilo:2017}.

We review cofibration categories in \cref{sec:cofib} and combine arguments of May and Schwede to establish the strong Verdier structure for biexact functors induced by biexact functors on stable cofibration categories in \cref{sec:cofib_biexact}. A further application is a new proof that the homotopy category of any symmetric monoidal stable model category is tensor triangulated. Moreover, we show that in addition to the derived monoidal product, the derived internal hom functor admits a strong Verdier structure as well; see \cref{htpy_category_of_monoidal_modelcat}.

An essential class of biexact functors on triangulated categories is the class of actions of a tensor triangulated category $(\cat{S},\otimes,\unit)$ on a triangulated category $\cat{T}$. If the action is induced by a biexact functor of stable cofibration categories, we say that the action is \emph{topological}. Any graded endomorphism of the unit $\unit$ in $\cat{S}$ induces a natural transformation on $\cat{T}$, which is compatible with the suspension. In fact, this yields a ring homomorphism from the graded endomorphism ring of $\unit$ in $\cat{S}$ to the center of $\cat{T}$. 

For a sequence of elements in the center, one can define a Koszul object. This generalizes the Koszul complex over a ring. For Koszul objects we show the following level inequality. Again, the bound can be achieved.

\begin{introthm} \label{intro:kosobj_level}
Let $\alpha_1, \ldots, \alpha_c$ be elements in the center of a topological triangulated category $\cat{T}$ that are induced by a topological action on $\cat{T}$. Then
\begin{equation*}
\level{\cat{T}}{X}{\kosobj{X}{(\alpha_1, \ldots, \alpha_c)}} \leq c + 1
\end{equation*}
for any $X \in \cat{T}$.
\end{introthm}
This result is contained in \cref{kos_obj_level} combined with \cref{cofib_biexact_verdier}. Koszul objects have been used to obtain bounds of Rouquier dimension in \cite{Bergh/Iyengar/Krause/Oppermann:2010}, and they are connected to support \cite[Section~5]{Benson/Iyengar/Krause:2008}. 

\begin{ack}
We thank Jan {\v{S}}{\v{t}}ov{\'\i}{\v{c}}ek for helpful discussions.

The authors were partly supported by the Deutsche Forschungsgemeinschaft (SFB-TRR 358/1 2023 - 491392403). Letz was also partly supported by the Alexander von Humboldt Foundation in the framework of a Feodor Lynen research fellowship endowed by the German Federal Ministry of Education and Research. Part of this work was done at the Hausdorff Research Institute for Mathematics, Bonn, when the authors were at the ``Spectral Methods in Algebra, Geometry, and Topology" trimester, funded by the Deutsche Forschungsgemeinschaft under Germany's Excellence Strategy--EXC-2047/1--390685813. 
\end{ack}

\section{Homotopy cartesian squares}

We fix a triangulated category $\cat{T}$ with suspension functor $\susp$. We consider $\op{\cat{T}}$ as a triangulated category in which $X \xrightarrow{\op{f}} Y \xrightarrow{\op{g}} Z \xrightarrow{\op{h}} \susp^{-1} X$ is an exact triangle in $\op{\cat{T}}$ if and only if $\susp^{-1}X \xrightarrow{-h} Z \xrightarrow{-g} Y \xrightarrow{-f} X$ is an exact triangle in $\cat{T}$. While this convention is not significant in this section, it ensures that the internal hom functor of a closed tensor triangulated category is biexact in \cref{monoidal_cofib_cat}.

\begin{discussion} \label{htpy_cartesian}
Following \cite[Definition~1.4.1]{Neeman:2001}, we call a commutative square
\begin{equation} \label{eq:square}
\begin{tikzcd}
T \ar[r,"g"] \ar[d,"f" swap] & V \ar[d,"f'"] \\
U \ar[r,"g'" swap] & X
\end{tikzcd}
\end{equation}
\emph{homotopy cartesian}, if there is an exact triangle
\begin{equation*}
T \xrightarrow{\left(\begin{smallmatrix}f \\ -g\end{smallmatrix}\right)} U \oplus V \xrightarrow{\left(\begin{smallmatrix}g' & f'\end{smallmatrix}\right)} X \xrightarrow{\partial} \susp T .
\end{equation*}
We say $\partial$ is a \emph{connecting morphism} of the homotopy cartesian square; other sources use the term \emph{differential}. 

The homotopy cartesian property is symmetric in that the square \cref{eq:square} is homotopy cartesian if and only if its reflection
\begin{equation*}
\begin{tikzcd}
T \ar[r,"f"] \ar[d,"g" swap] & U \ar[d,"g'"] \\
V \ar[r,"f'" swap] & X
\end{tikzcd}
\end{equation*}
is homotopy cartesian. However, if $\partial$ a connecting morphism of the homotopy cartesian square \cref{eq:square}, then $-\partial$ is a connecting morphism of the reflected square.

Further, the homotopy cartesian property is self-dual; that is, a square \cref{eq:square} is homotopy cartesian in $\cat{T}$ if and only if its dual is homotopy cartesian in $\op{\cat{T}}$. 
\end{discussion}

Homotopy cartesian squares are strongly connected to morphisms of triangles in which one morphism is the identity morphism. 

\begin{discussion} \label{htpy_cart_mor_triangle}
Given a morphism of exact triangles
\begin{equation*}
\begin{tikzcd}
T \ar[r] \ar[d] & V \ar[r] \ar[d] & Z \ar[r] \ar[d,"="] & \susp T \ar[d] \\
U \ar[r] & X \ar[r] & Z \ar[r] & \susp U
\end{tikzcd}
\end{equation*}
where the first square is homotopy cartesian, a connecting morphism is given by the composite
\begin{equation*}
\partial = (X \to Z \to \susp T) .
\end{equation*}

If either the morphism $X \to Z$ or $Z \to \susp T$ is not given, it can be constructed such that the diagram is a morphism of triangles; see \cite[Lemma~1.4.4]{Neeman:2001}. Similarly, if either the morphism $V \to X$ or $T \to U$ is not given, it can be constructed such that the diagram commutes and the first square is homotopy cartesian; see \cite[Lemma~1.4.3]{Neeman:2001}.
\end{discussion}

While homotopy cartesian squares are in general neither pullback nor pushout squares, they still satisfy a weakened pasting property.

\begin{discussion} \label{combiningsquares}
Consider two commutative squares and their composition:
\begin{center}
\begin{subequations}
\noindent\begin{minipage}{.33\linewidth}
\begin{equation} \label{eq:left_square}
\begin{tikzcd}
T \ar[r,"g"] \ar[d,"f" swap] & V \ar[d,"f'"] \\
U \ar[r,"g'" swap] & X
\end{tikzcd}
\end{equation}
\end{minipage}%
\begin{minipage}{.33\linewidth}
\begin{equation} \label{eq:right_square}
\begin{tikzcd}
V \ar[r,"h"] \ar[d,"f'" swap] & Y \ar[d,"f''"] \\
X \ar[r,"h'" swap] & Z
\end{tikzcd}
\end{equation}
\end{minipage}%
\begin{minipage}{.33\linewidth}
\begin{equation} \label{eq:outside_square}
\begin{tikzcd}
T \ar[r,"hg"] \ar[d,"f" swap] & Y \ar[d,"f''"] \\
U \ar[r,"h'g'" swap] & Z
\end{tikzcd}
\end{equation}
\end{minipage}%
\end{subequations}
\end{center}
\begin{enumerate}
\item \label{combiningsquares:outside} If \cref{eq:left_square,eq:right_square} are homotopy cartesian, then \cref{eq:outside_square} is homotopy cartesian; see \cite[Proposition~6.11]{Christensen/Frankland:2022}.
\item \label{combiningsquares:right} If \cref{eq:left_square,eq:outside_square} are homotopy cartesian, then there exists a morphism $X \xrightarrow{\tilde{h}'} Z$, such that the square
\begin{equation*}
\begin{tikzcd}
V \ar[d,"f'"] \ar[r,"h"] & Y \ar[d,"f''"] \\
X \ar[r,"\tilde{h}'" swap] & Z
\end{tikzcd}
\end{equation*}
is homotopy cartesian, and $h' g' = \tilde{h}' g'$; see \cite[Lemma~9]{Saorin/Zimmermann:2019}. 
\item \label{combiningsquares:left} If \cref{eq:right_square,eq:outside_square} are homotopy cartesian, then there exists a morphism $T \xrightarrow{\tilde{g}} V$ such that the square
\begin{equation*}
\begin{tikzcd}
T \ar[r,"\tilde{g}"] \ar[d,"f" swap] & V \ar[d,"f'"] \\
U \ar[r,"g'" swap] & X
\end{tikzcd}
\end{equation*}
is homotopy cartesian, and $hg = h \tilde{g}$; this is (2) in $\op{\cat{T}}$. 
\end{enumerate}

If the squares \cref{eq:left_square,eq:right_square,eq:outside_square} are homotopy cartesian, then there exist connecting morphisms $\partial_1$, $\partial_2$ and $\partial_3$, respectively, such that
\begin{equation*}
\partial_3 h' = \partial_1 \quad\text{and}\quad (\susp f) \partial_3 = \partial_2 .
\end{equation*}
\end{discussion}

\subsection{Homotopy pushouts} \label{htpy_pushout}

If \cref{eq:square} is homotopy cartesian, we say $X$ is the \emph{homotopy pushout} of the span $U \gets T \to V$, and we write $X = U +_T V$. A homotopy pushout is equipped with morphisms $U \to U +_T V$ and $V \to U +_T V$. Just as the cone of a morphism, the homotopy pushout is unique up to non-canonical isomorphism; in fact we have $U +_T 0 = \cone(T \to U)$. 

We emphasize that a homotopy pushout is typically not a pushout. In fact given morphisms $U \to Z$ and $V \to Z$ that coincide after pre-composition with $T \to U$ and $T \to V$, respectively, with each other, there exists a \emph{non-unique} morphism $U +_T V \to Z$. Later we investigate situations in which there exists a morphism $U +_T V \to Z$ with a cone compatible with a given $3 \times 3$ diagram containing the square with $T$, $U$, $V$ and $Z$. We say such situations ``admit a (strong) Verdier structure''; see \cref{verdier_structure}. 

\begin{discussion} \label{expandSum}
We take the homotopy pushout of spans $U \gets S \to V$ and $V \gets T \to W$, and obtain the commutative diagram
\begin{equation*}
\begin{tikzcd}
& T \ar[r] \ar[d] & W \ar[d] \\
S \ar[r] \ar[d] & V \ar[r] \ar[d] & V +_T W \nospacepunct{.} \\
U \ar[r] & U +_S V
\end{tikzcd}
\end{equation*}
In this diagram we complete the square in the lower right corner to a homotopy cartesian square. Then all rectangles, in particular the horizontal and vertical one, are homotopy cartesian by \cref{combiningsquares}, and we obtain
\begin{equation*}
(U +_S V) +_T W = (U +_S V) +_V (V +_T W) = U +_S (V +_T W) .
\end{equation*}
That is, the construction of the homotopy pushout is associative.
\end{discussion}

\begin{lemma} \label{htpy_pushout_same_base}
For any diagram $X \leftarrow U \leftarrow S \rightarrow V \rightarrow Y$ there exists a morphism $U +_S V \to X +_S Y$ such that its cone is $\cone(U \to X) \oplus \cone(V \to Y)$. 
\end{lemma}
\begin{proof}
We consider the commutative diagram
\begin{equation*}
\begin{tikzcd}
S \ar[d,"="] \ar[r] & U \oplus V \ar[d] \ar[r] & U +_S V \ar[d,dashed] \ar[r] & \susp S \ar[d,"="] \\
S \ar[r] & X \oplus Y \ar[r] & X +_S Y \ar[r] & \susp S
\end{tikzcd}
\end{equation*}
in which the rows are exact triangles. By \cref{htpy_cart_mor_triangle}, the dashed arrow exists such that the diagram is a morphism of exact triangles and the second square is homotopy cartesian. Then the dashed arrow is the desired morphism and, by \cref{htpy_cart_mor_triangle}, it has the desired cone. 
\end{proof}

Alternatively, we can change the base of the span. 

\begin{lemma} \label{htpy_pushout_map_base}
Given $S \to T$ and a span $X \leftarrow T \rightarrow Y$, then there exists a morphism $X +_S Y \to X +_T Y$ such that its cone is $\susp \cone(S \to T)$.
\end{lemma}
\begin{proof}
We consider the commutative diagram
\begin{equation*}
\begin{tikzcd}
S \ar[d] \ar[r] & X \oplus Y \ar[d,"="] \ar[r] & X +_S Y \ar[d,dashed] \ar[r] & \susp S \ar[d] \\
T \ar[r] & X \oplus Y \ar[r] & X +_T Y \ar[r] & \susp T
\end{tikzcd}
\end{equation*}
in which the rows are exact triangles. By \cref{htpy_cart_mor_triangle}, the dashed arrow exists such that the diagram is a morphism of exact triangles and the third square is homotopy cartesian. Then the dashed arrow is the desired morphism and, by \cref{htpy_cart_mor_triangle}, it has the desired cone. 
\end{proof}

\begin{remark} \label{construct_comp_mor}
Given a commutative diagram
\begin{equation*}
\begin{tikzcd}
T \ar[r] \ar[d] & V \ar[d] \\
U \ar[r] & X
\end{tikzcd}
\end{equation*}
there exists a morphism $U +_T V \to X$ such that the compositions $U \to U +_T V \to X$ and $V \to U +_T V \to X$ recover the given morphisms. Using \cref{htpy_cart_mor_triangle} we obtain homotopy cartesian squares
\begin{equation*}
\begin{tikzcd}
U +_T V \ar[r] \ar[d] & \cone(T \to V) \ar[d,dashed] \\
X \ar[r] & \cone(U \to X)
\end{tikzcd}
\quad \text{and} \quad 
\begin{tikzcd}
U +_T V \ar[r] \ar[d] & \cone(T \to U) \ar[d,dashed] \\
X \ar[r] & \cone(V \to X) \nospacepunct{.}
\end{tikzcd}
\end{equation*}
In particular, this yields that the dashed vertical arrows have the same cone by \cref{htpy_cart_mor_triangle}. 

In general, morphisms $\cone(T \to V) \to \cone(U \to X)$ and $\cone(T \to U) \to \cone(V \to X)$ that are compatible with the commutative square need not have the same cone. In the next section we discuss a setting in which there is a natural choice for these morphisms, and there exists a morphism $U +_T V \to X$ that is compatible with this natural choice.
\end{remark}

\section{Biexact functors between triangulated categories}

In this section we show the level inequality for a biexact functor $\mathsf{F}$. If we apply the biexact functor $\mathsf{F}$ to a triangle in each component, we obtain a $3 \times 3$ diagram in which each row and each column is an exact triangle. In general, this is not enough to be able to construct a compatible morphism as discussed in \cref{construct_comp_mor}. This will be resolved by the notion of a strong Verdier structure.

\begin{discussion} \label{biexact}
Recall that an exact functor $\mathsf{F} \colon \cat{S} \to \cat{T}$ of triangulated categories is equipped with a natural isomorphism $\tau \colon \mathsf{F} \susp \to \susp \mathsf{F}$ such that for any exact triangle $X \xrightarrow{f} Y \xrightarrow{g} Z \xrightarrow{h} \susp X$ in $\cat{S}$ the triangle
\begin{equation*}
\mathsf{F}(X) \xrightarrow{\mathsf{F}(f)} \mathsf{F}(Y) \xrightarrow{\mathsf{F}(g)} \mathsf{F}(Z) \xrightarrow{\tau \mathsf{F}(h)} \susp \mathsf{F}(X)
\end{equation*}
is exact in $\cat{T}$. 

We call a bifunctor $\mathsf{F} \colon \cat{S}\times\cat{T}\to \cat{U}$ of triangulated categories \emph{biexact}, if it is equipped with natural isomorphisms
\begin{equation*}
\theta \colon \mathsf{F}(\susp -,-) \to \susp \mathsf{F}(-,-) \quad \text{and} \quad \zeta \colon \mathsf{F}(-,\susp -) \to \susp \mathsf{F}(-,-)
\end{equation*}
such that
\begin{enumerate}
\item for any $X \in \cat{S}$ the functor $\mathsf{F}(X,-)$ with $\zeta(X,-)$ is exact,
\item for any $X' \in \cat{T}$ the functor $\mathsf{F}(-,X')$ with $\theta(-,X')$ is exact, and
\item the following square anti-commutes
\begin{equation} \label{biexact:anticommute}
\begin{tikzcd}
\mathsf{F}(\susp X,\susp X') \ar[r, "\theta"] \ar[d, "\zeta" swap] \ar[dr,phantom,"(-1)"] & \susp \mathsf{F}(X,\susp X') \ar[d, "\susp\zeta"] \\
\susp \mathsf{F}(\susp X,X') \ar[r, "\susp \theta" swap] & \susp^2 \mathsf{F}(X,X') \nospacepunct{.}
\end{tikzcd}
\end{equation}
\end{enumerate}
In \cite[Definition~10.3.6]{Kashiwara/Schapira:2006} such a functor is called a triangulated bifunctor. 
\end{discussion}

\begin{discussion} \label{verdier_structure}
Following \cite[Theorem~3.5]{Keller/Neeman:2002} and \cite[Definition~3.31]{Aldrovandi/Lester:2023}, we say a biexact functor $\mathsf{F} \colon \cat{S}\times\cat{T}\to \cat{U}$ \emph{admits a Verdier structure}, if for all exact triangles
\begin{equation} \label{biexact:setup}
X \xrightarrow{f} Y \xrightarrow{g} Z \xrightarrow{h} \susp X \quad\text{and}\quad X' \xrightarrow{f'} Y' \xrightarrow{g'} Z' \xrightarrow{h'} \susp X'
\end{equation}
in $\cat{S}$ and $\cat{T}$, respectively, there exists an object $W \in \cat{U}$ and exact triangles
\begin{equation} \label{biexact:exact-triangles}
\begin{tikzcd}[row sep=0,column sep=large]
\mathsf{F}(X,Y') \ar[r,"j"] & W \ar[r,"q"] & \mathsf{F}(Z,X') \ar[r,dotted,"{\theta \mathsf{F}(h,f')}"] & \susp \mathsf{F}(X,Y') \\
W \ar[r,"i"] & \mathsf{F}(Y,Y') \ar[r,dotted,"{\mathsf{F}(g,g')}"] & \mathsf{F}(Z,Z') \ar[r,"p"] & \susp W \\
\mathsf{F}(Y,X') \ar[r,"j'"] & W \ar[r,"q'"] & \mathsf{F}(X,Z') \ar[r,dotted,"\zeta {\mathsf{F}(f,h')}"] & \susp \mathsf{F}(Y,X') \nospacepunct{,}
\end{tikzcd}
\end{equation}
such that the following diagram commutes
\begin{equation}\label{eq:bifunctorVerdier}
\begin{tikzcd}[column sep=-0.5em,row sep=0]
\mathsf{F}(X,X') \ar[rrr] \ar[ddd] \ar[ddrr,phantom,"\textrm{(I)}"] &&& \mathsf{F}(Y,X') \ar[rrr,"\textrm{(iv)}" swap] \ar[ddd,"\textrm{(iii)}"] \ar[ddl] \ar[dr] &&& \mathsf{F}(Z,X') \ar[rrr] \ar[ddd] \ar[dddrrr,dotted] &&& \susp \mathsf{F}(X,X') \ar[ddd] \\[+0.5em]
&&& & W \ar[urr] \ar[ddl] \ar[ddrr,phantom,"\textrm{(II)}"] \\
&& W \ar[dr] \\[+0.5em]
\mathsf{F}(X,Y') \ar[rrr,"\textrm{(ii)}" swap] \ar[ddd,"\textrm{(i)}"] \ar[urr] \ar[dr] &&& \mathsf{F}(Y,Y') \ar[rrr] \ar[ddd] \ar[dddrrr,dotted] &&& \mathsf{F}(Z,Y') \ar[rrr] \ar[ddd] \ar[ddrr,phantom,"\textrm{(IV)}"] &&& \susp \mathsf{F}(X,Y') \ar[ddd] \ar[ddl] \\[+0.5em]
& W \ar[urr] \ar[ddl] \ar[ddrr,phantom,"\textrm{(III)}"] \\
&&& &&& && \susp W \ar[dr] \\[+0.5em]
\mathsf{F}(X,Z') \ar[rrr] \ar[ddd] \ar[dddrrr,dotted] &&& \mathsf{F}(Y,Z') \ar[rrr] \ar[ddd] \ar[ddrr,phantom,"\textrm{(V)}"] &&& \mathsf{F}(Z,Z') \ar[rrr,"\textrm{(vi)}" swap] \ar[ddd,"\textrm{(v)}"] \ar[urr] \ar[dr] \ar[ddl] &&& \susp \mathsf{F}(X,Z') \ar[ddd] \\[+0.5em]
&&& &&& & \susp W \ar[urr] \ar[ddl] \ar[ddrr,phantom,"\textrm{(VI)}",shift left=2em] \\
&&& && \susp W \ar[dr] & && (-1) \\[+0.5em]
\susp \mathsf{F}(X,X') \ar[rrr] &&& \susp \mathsf{F}(Y,X') \ar[rrr] \ar[urr] &&& \susp \mathsf{F}(Z,X') \ar[rrr] &&& \susp^2 \mathsf{F}(X,X') \nospacepunct{.}
\end{tikzcd}
\end{equation}
In the diagram all rows and columns are the exact triangles obtained by applying $\mathsf{F}$ to \cref{biexact:exact-triangles}. The anti-commutativity of the bottom right square is due to \cref{biexact:anticommute}. The morphisms involving $W$ are those that appear in the triangles \cref{biexact:exact-triangles}, or the suspension of those; this means for example that the morphism $\susp W \to \susp \mathsf{F}(X,Z')$ is the morphism $\susp q'$. 

Abridged, the commutativity of \cref{eq:bifunctorVerdier} means the squares (I--VI) and the triangles (i--vi) commute.

Further, we say $\mathsf{F}$ admits a \emph{strong} Verdier structure, if the squares (I--VI) are homotopy cartesian; for (VI) we mean it is homotopy cartesian after we replace any one of the morphisms by its negative. 

If $\mathsf{F}$ admits a strong Verdier structure, then the connecting morphisms are given by the appropriate compositions of morphisms in \cref{biexact:exact-triangles}; cf.\@ \cref{htpy_cart_mor_triangle}. 
\end{discussion}

\begin{remark}
If a biexact functor $\mathsf{F}$ admits a (strong) Verdier structure, then the diagram \cref{eq:bifunctorVerdier} can be extended in any direction by rotation. In particular, a biexact functor $\mathsf{F}$ admits a (strong) Verdier structure if and only if $\mathsf{F}(\susp -,-)$, or equivalently $\mathsf{F}(-,\susp -)$, does. 

A bifunctor $\mathsf{F} \colon \cat{S}\times\cat{T}\to \cat{U}$ of triangulated categories is biexact if and only if its opposite $\op{\mathsf{F}} \colon \op{\cat{S}}\times\op{\cat{T}}\to \op{\cat{U}}$ is biexact. Moreover, a biexact functor $\mathsf{F}$ admits a (strong) Verdier structure if and only if $\op{\mathsf{F}}$ does.
\end{remark}

The symmetric monoidal product of a tensor triangulated category is a special case of a biexact functor. In this situation \cref{verdier_structure} is (TC3) of \cite{May:2001}. May provided an informal argument for establishing this condition if the tensor triangulated category comes from a symmetric monoidal, stable model category and checked it for certain families of such model categories. In \cite{Groth/Ponto/Shulman:2014} the same condition was shown for tensor triangulated categories induced by a stable monoidal derivator. In \cref{sec:cofib_biexact} we show that any biexact functor between stable cofibration categories induces a biexact functor admitting a strong Verdier structure on homotopy categories. In particular, this includes many bifunctors on algebraic triangulated categories; for examples see \cref{bifunctor_bimodules}.

\begin{remark}
For any biexact functor $\mathsf{F} \colon \cat{S} \times \cat{T} \to \cat{U}$ and exact triangle $X \to Y \to Z \to \susp X$ in $\cat{S}$, any morphism $X' \to Y'$ in $\cat{T}$ induces a morphism of exact triangles in $\cat{U}$. This morphism of exact triangles is middling good in the sense of \cite[Definition~2.4]{Neeman:1991}. Moreover, if $\mathsf{F}$ admits a Verdier structure, then the morphism of exact triangles is Verdier good in the sense of \cite[Definition~3.1]{Christensen/Frankland:2022}.
\end{remark}

Before we establish the level inequality for a biexact functor, we recall the definition of level.

\begin{discussion}
For an object $X$ of $\cat{T}$, we denote by $\thick[0]{\cat{T}}{X}$ the full subcategory of zero objects. We denote the smallest full subcategory that contains $X$ and is closed under (de)suspension, finite coproducts, and retracts by $\thick[1]{\cat{T}}{X}$. For $n \geq 2$ we inductively set
\begin{equation*}
\thick[n]{\cat{T}}{X} \coloneqq \set{Y \in \cat{T}}{\begin{gathered} \text{there is an exact triangle} \\ Y' \to Y \oplus \tilde{Y} \to Y'' \to \susp Y' \\ \text{with } Y' \in \thick[1]{\cat{T}}{X} \text{ and } Y'' \in \thick[n-1]{\cat{T}}{X}\end{gathered}} .
\end{equation*}
The subcategories give an exhaustive filtration of the smallest thick subcategory containing $X$. The construction of these subcategories is robust: In the definition we may assume that $Y' \in \thick[i]{\cat{T}}{X}$ and $Y'' \in \thick[j]{\cat{T}}{X}$ for any $i+j = n$. Moreover, if $Y \in \thick[n]{\cat{T}}{X}$, then it is enough to take a retract in the last step.

The subcategories $\thick[n]{\cat{T}}{X}$ were first introduced in \cite[2.2]{Bondal/VanDenBergh:2003} where they were denoted by $\langle X \rangle_n$. With the notation $\thick[n]{\cat{T}}{X}$ we follow \cite[2.2]{Avramov/Buchweitz/Iyengar/Miller:2010}.
\end{discussion}

Let $X, Y \in \cat{T}$. The \emph{$X$-level of $Y$} is
\begin{equation*}
\level{\cat{T}}{X}{Y} \coloneqq \inf\set{n \geq 0}{Y \in \thick[n]{\cat{T}}{X}} ;
\end{equation*}
see \cite[2.3]{Avramov/Buchweitz/Iyengar/Miller:2010}. 

\begin{theorem} \label{biexact_level}
Let $\mathsf{F} \colon \cat{S} \times \cat{T} \to \cat{U}$ be a biexact functor that admits a strong Verdier structure. Then
\begin{equation*}
\level{\cat{U}}{\mathsf{F}(X,X')}{\mathsf{F}(Y,Y')} \leq \level{\cat{S}}{X}{Y} + \level{\cat{T}}{X'}{Y'} - 1
\end{equation*}
for any $X, Y$ in $\cat{S}$ and $X', Y'$ in $\cat{T}$. 
\end{theorem}
\begin{proof}
For $i,j \geq 1$, let
\begin{equation} \label{eq:setupBuilding}
\begin{tikzcd}[row sep=0em,column sep=small]
Y_{i-1} \ar[r] & Y_i \ar[r] & X_i \ar[r] & \susp Y_{i-1} \\
Y'_{j-1} \ar[r] & Y'_j \ar[r] & X'_j \ar[r] & \susp Y'_{j-1}
\end{tikzcd} \quad\text{with}\quad
\begin{tikzcd}[row sep=0em,column sep=0.2em]
X_i \in \thick[1]{\cat{S}}{X} \\
X'_i \in \thick[1]{\cat{T}}{X'}
\end{tikzcd}
\end{equation}
be exact triangles in $\cat{S}$ and $\cat{T}$, respectively, with $Y_0 = 0$ and $Y'_0 = 0$. We set $Z_{i,j} \coloneqq \mathsf{F}(Y_i,Y'_j)$. Since $\mathsf{F}$ admits a strong Verdier structure, there exists a homotopy pushout of the span $Z_{i-1,j} \leftarrow Z_{i-1,j-1} \rightarrow Z_{i,j-1}$ that fits into an exact triangle
\begin{equation*}
(Z_{i-1,j} +_{Z_{i-1,j-1}} Z_{i,j-1}) \to Z_{i,j} \to \mathsf{F}(X_i,X'_j) \to \susp (Z_{i-1,j} +_{Z_{i-1,j-1}} Z_{i,j-1})
\end{equation*}
and makes the following diagram commute
\begin{equation} \label{biexact_level:mor_compatible}
\begin{tikzcd}
Z_{i-1,j-1} \ar[r] \ar[d] & Z_{i,j-1} \ar[d] \ar[ddr,bend left=15] \\
Z_{i-1,j} \ar[r] \ar[rrd,bend right=15] & Z_{i-1,j} +_{Z_{i-1,j-1}} Z_{i,j-1} \ar[dr] \\
&& Z_{i,j} \nospacepunct{.}
\end{tikzcd}
\end{equation}

We let
\begin{equation*}
W_k \coloneqq Z_{1,k} +_{Z_{1,k-1}} Z_{2,k-1} +_{Z_{2,k-2}} \dots +_{Z_{k-1,1}} Z_{k,1}\quad \text{for } k \geq 1
\end{equation*}
be the iterated homotopy pushout over the diagonal $i+j=k+1$. By the associativity of the homotopy pushout we can write
\begin{align*}
W_k &= Z_{1,k} +_{Z_{1,k}} (Z_{1,k} +_{Z_{1,k-1}} Z_{2,k-1}) +_{Z_{2,k-1}} \ldots \\
&\quad\quad +_{Z_{k-1,2}} (Z_{k-1,2} +_{Z_{k-1,1}} Z_{k,1}) +_{Z_{k,1}} Z_{k,1} .
\end{align*}
Expressed in this way $W_k$ is an iterated homotopy pushout over the same bases as $W_{k+1}$. Since \cref{biexact_level:mor_compatible} commutes, we can inductively apply \cref{htpy_pushout_same_base} to obtain a morphism $W_k \to W_{k+1}$ with
\begin{equation*}
\cone(W_k \to W_{k+1}) = \bigoplus_{i+j=k+2} \mathsf{F}\left(X_i,X'_j\right) \in \thick[1]{\cat{U}}{\mathsf{F}(X,X')} .
\end{equation*}
As $W_1 = \mathsf{F}(X_1,X'_1) \in \thick[1]{\cat{U}}{\mathsf{F}(X,X')}$, it follows that
\begin{equation*}
\level{\cat{U}}{\mathsf{F}(X,X')}{W_k} \leq k .
\end{equation*}

Now we use this inequality to establish the claim. We may assume that $m \coloneqq \level{\cat{S}}{X}{Y}$ and $n \coloneqq \level{\cat{T}}{X'}{Y'}$ are finite. Then there exist sequences of triangles \cref{eq:setupBuilding}, such that $Y$ is a retract of $Y_m$, and $Y'$ is a retract of $Y'_n$, and the sequences stabilize afterwards: $Y_i = Y_m$ and $X_{i+1} = 0$ for $i \geq m$ and $Y'_j = Y'_n$ and $X'_{j+1} = 0$ for $j \geq n$. In particular, the morphisms $Z_{i-1,j} \to Z_{i,j}$ are identities for $i>m$ and the morphisms $Z_{i,j-1}\to Z_{i,j}$ are identities for $j>n$. Thus the iterated homotopy pushout $W_{m+n-1}$ simplifies to $W_{m+n-1} = \mathsf{F}(Y_m,Y'_n)$. It follows that $\level{\cat{U}}{\mathsf{F}(X,X')}{\mathsf{F}(Y,Y')} \leq m+n-1$. 
\end{proof}

We will treat cofibration categories in \cref{sec:cofib} to provide a large class of examples of biexact functors admitting a strong Verdier structure in \cref{sec:cofib_biexact} and hence where \cref{biexact_level} holds. Concrete examples are given in \cref{sec:cofib_biexact,sec:examples}. First we establish another level inequality in a closely related setting. 

\section{Koszul objects}

Koszul objects generalize Koszul complexes to triangulated categories. They are defined for a sequence of elements in a ring that acts on the triangulated category. 

\begin{discussion}
The \emph{(graded) center} of a triangulated category $\cat{T}$ is
\begin{equation*}
\ctr^*(\cat{T}) \coloneqq \bigoplus_{d \in \BZ} \set{\alpha \colon \id_\cat{T} \to \susp^d}{\alpha \susp = (-1)^d \susp \alpha} .
\end{equation*}
This is a graded-commutative graded ring; that means
\begin{equation*}
(\susp^{|\alpha|} \beta) \circ \alpha = (-1)^{|\alpha||\beta|} (\susp^{|\beta|} \alpha) \circ \beta ;
\end{equation*}
see \cite{Buchweitz/Flenner:2008}. An action of a graded-commutative graded ring $R$ on $\cat{T}$ is equivalent to a graded ring homomorphism $R \to \ctr^*(\cat{T})$. 
\end{discussion}

\begin{discussion}
For $X \in \cat{T}$ and a sequence $\bmalpha = \alpha_1, \ldots, \alpha_c$ in $\ctr^*(\cat{T})$, the \emph{Koszul object of $\bmalpha$ on $X$} is
\begin{equation*}
\kosobj{X}{\bmalpha} \coloneqq 
\begin{cases}
X & c = 0 \\
\cone(X \xrightarrow{\alpha_1(X)} \susp^{|\alpha_1|} X) & c = 1 \\
\kosobj{(\kosobj{X}{(\alpha_1, \ldots, \alpha_{c-1})})}{\alpha_c} & c > 1 .
\end{cases}
\end{equation*}
The Koszul object is unique up to non-unique isomorphism.
\end{discussion}

From the construction of the Koszul object we immediately get the inequality 
\begin{equation*}
\level{\cat{T}}{X}{\kosobj{X}{(\alpha_1, \ldots, \alpha_c)}} \leq 2^c.
\end{equation*}
We can improve this bound when the elements $\alpha_i$ arise from an action of a monoidal triangulated category on $\cat{T}$.

\subsection{Action by monoidal categories}

We recall the definition and fix notation for a monoidal structure on a triangulated category; see \cite[Definition~A.2.1]{Hovey/Palmieri/Strickland:1997}. For details on the coherence axioms see \cite[XI.1]{MacLane:1998} and also \cite{Kelly:1964}.

\begin{discussion} \label{monoidal_triangulated}
A \emph{monoidal triangulated category} $(\cat{S},\otimes,\unit)$ consists of a triangulated category $\cat{S}$ with a monoidal product $(\otimes,\unit)$ where $\otimes\colon \cat{S}\times \cat{S}\to \cat{S}$ is a biexact functor and $\unit \in \cat{S}$ the \emph{unit} object. This means there are isomorphisms
\begin{equation*}
\alpha \colon X \otimes (Y \otimes Z) \to (X \otimes Y) \otimes Z  ,\quad \lambda \colon \unit \otimes X \to X \quad \text{and} \quad \rho \colon X \otimes \unit \to X ,
\end{equation*}
that are natural transformations of exact functors in each variable satisfying the coherence axioms of a monoidal category.

A monoidal triangulated category $(\cat{S},\otimes,\unit)$ is \emph{symmetric} if it is equipped with an isomorphism
\begin{equation*}
\sigma \colon X \otimes Y \to Y \otimes X
\end{equation*}
that is a natural transformation of exact functors in each variable satisfying the coherence axioms of a symmetric monoidal category. Symmetric monoidal triangulated categories are also called \emph{tensor triangulated categories}. Some sources additionally assume that the monoidal structure is closed; we discuss closed monoidal structures on triangulated categories in \cref{symmetricmonoidalclosed}.
\end{discussion}

\begin{discussion}
Let $(\cat{S},\otimes,\unit)$ be a monoidal triangulated category. A \emph{ (left) action} of $\cat{S}$ on a triangulated category $\cat{T}$ consists of a biexact bifunctor $\mathsf{F}\colon \cat{S} \times \cat{T} \to \cat{T}$ together with isomorphisms
\begin{equation*}
\alpha \colon \mathsf{F}(X,\mathsf{F}(Y,Z)) \to \mathsf{F}(X \otimes Y,Z) \quad \text{and} \quad \lambda \colon \mathsf{F}(\unit,Z) \to Z ,
\end{equation*}
that are natural transformations of exact functors in each variable satisfying coherence axioms analogous to the ones of a monoidal category; see for example \cite[Definition~4.1.6]{Hovey:1999} and \cite[Section~1]{Buan/Krause/Snashall/Solberg:2020} for details. In particular, any monoidal triangulated category $(\cat{S},\otimes,\unit)$ acts on itself via the tensor product $\otimes\colon \cat{S}\times \cat{S}\to \cat{S}$.
\end{discussion}

\begin{discussion}
Let $(\cat{S},\otimes,\unit)$ be a monoidal triangulated category and $\mathsf{F}$ an action of $\cat{S}$ on a triangulated category $\cat{T}$. We denote by
\begin{equation*}
\End[*]{\cat{S}}{X} \coloneqq \bigoplus_{n \in \BZ} \cat{S}(X,\susp^n X)
\end{equation*}
the graded endomorphism ring. Then the monoidal structure induces a homomorphism of graded rings
\begin{equation*}
\End[*]{\cat{S}}{\unit} \to \ctr^*(\cat{T})  ,\quad f \mapsto \alpha_f \coloneqq (X \cong \mathsf{F}(\unit,X) \xrightarrow{\mathsf{F}(f,X)} \mathsf{F}(\susp^{|f|} \unit,X) \cong \susp^{|f|} X) ;
\end{equation*}
see for example \cite[Proposition~2.1]{Buan/Krause/Snashall/Solberg:2020}. The isomorphisms involve $\lambda$ from the $\cat{S}$-action and $\theta$ from \cref{biexact}. We say $\alpha \in \ctr^*(\cat{T})$ is \emph{induced by $\mathsf{F}$}, if $\alpha = \alpha_f$ for some $f \in \End{\cat{S}}{\unit}$. We can choose
\begin{equation} \label{kosobj_can_choice}
\kosobj{X}{\alpha_f} = \mathsf{F}(\cone(f),X)
\end{equation}
as the Koszul object of $\alpha_f$ on $X$; in particular the Koszul object is functorial in $X$ for $\alpha_f$. Moreover, for $f,g \in \End[*]{\cat{S}}{\unit}$ we have an isomorphism
\begin{equation*}
\begin{aligned}
\kosobj{X}{(\alpha_f,\alpha_g)} &\cong \mathsf{F}(\cone(g),\mathsf{F}(\cone(f),X)) \cong \mathsf{F}(\cone(g) \otimes \cone(f),X) \\
&\cong \mathsf{F}(\cone(f) \otimes \cone(g),X) \cong \kosobj{X}{(\alpha_g,\alpha_f)} .
\end{aligned}
\end{equation*}
If the monoidal product is symmetric, then this isomorphism is induced by the natural isomorphism $\sigma$. Otherwise, the isomorphism $\cone(g) \otimes \cone(f) \cong \cone(f) \otimes \cone(g)$ is not canonical, though its existence follows from the $3 \times 3$ diagram obtained by applying $\otimes$ to the triangles involving $f$ and $g$.
\end{discussion}

In general, the objects $\kosobj{X}{(\alpha,\beta)}$ and $\kosobj{X}{(\beta,\alpha)}$ need not be isomorphic, as illustrated by the following example.

\begin{example} \label{counter_koszul_object}
Let $k$ be a field and $A = k[x]/(x^2)$. By \cite[Proposition~5.4]{Krause/Ye:2011}, the center of $\dbcat{\mod{A}}$ is the trivial extension ring $k[\zeta] \ltimes \prod_{r \geqslant 0} k$ where $\zeta$ is of degree $2$ if the characteristic of $k$ is not $2$ and of degree $1$ if $k$ is of characteristic $2$, and $k = k[\zeta]/(\zeta)$ as a $k[\zeta]$-module. The ring $k[\zeta]$ is the Hochschild cohomology of $A$ over $k[x]$ and the elements of $k[\zeta]$ are induced by a bifunctor; we discuss the action of Hochschild cohomology in \cref{sec:hochschild}. We focus on the elements in the center coming from $\prod_{r \geqslant 0} k$. The category $\dbcat{\mod{A}}$ is a Krull-Schmidt category where the indecomposable objects are the complexes $A_m^n$ that have $A$ in degrees $m \leq d \leq n$ connected by differentials $x\id_A$ and are zero elsewhere. We take $\eta_0, \eta_1 \in \prod_{r \geqslant 0} k$ determined by
\begin{equation*}
\eta_r(A_m^n) = \begin{cases} 
x_m^n & n-m = r \\
0 & n-m \neq 0 ,
\end{cases}
\end{equation*}
where $x_m^n$ is the morphism with $(x_m^n)_m = x \id_A$ and zero otherwise. This morphism is not homotopic to zero or $\id_A$. For these objects we obtain
\begin{equation*}
\kosobj{A_n^n}{\eta_0} = A_n^{n+1} \quad \text{and} \quad \kosobj{A_n^n}{\eta_1} = A_n^n \oplus \susp A^n_n .
\end{equation*}
In particular we obtain 
\begin{equation*}
\kosobj{A_n^n}{(\eta_0,\eta_1)} = A_n^{n+2} \oplus A_{n+1}^{n+1} \not\cong A_n^{n+1} \oplus \susp A_n^{n+1} = \kosobj{A^n_n}{(\eta_1,\eta_0)} .
\end{equation*}
Moreover, we have
\begin{equation*}
\kosobj{A_n^n}{(\eta_0+\eta_1)} = A_n^{n+1} \quad \text{and} \quad (\eta_0 + \eta_1)(\kosobj{A_n^n}{(\eta_0+\eta_1)}) = x_n^{n+1} \neq 0 .
\end{equation*}
Thus elements in the center of $\cat{T}$ are not necessarily trivial on their Koszul objects.
\end{example}

\subsection{Level inequalities involving Koszul objects}

For Koszul objects, that are induced by an action, we obtain:

\begin{lemma} \label{level_easy_kos}
Let $\cat{T}$ be a triangulated category and $\bmalpha = \alpha_1, \ldots, \alpha_c$ a sequence of elements in $\ctr^*(\cat{T})$ each induced by an action of a monoidal triangulated category $\mathsf{F}_i \colon \cat{S}_i \times \cat{T} \to \cat{T}$. Then
\begin{equation*}
\level{\cat{T}}{\kosobj{X}{\bmalpha}}{\kosobj{Y}{\bmalpha}} \leq \level{\cat{T}}{X}{Y}
\end{equation*}
for any objects $X$ and $Y$ in $\cat{T}$.
\end{lemma}
\begin{proof}
Let $\alpha = \alpha_f$ be an element induced by an action $\mathsf{F}$. Then by \cref{kosobj_can_choice} we have
\begin{equation*}
\level{\cat{T}}{\kosobj{X}{\alpha}}{\kosobj{Y}{\alpha}} = \level{\cat{T}}{\mathsf{F}(\cone(f),X)}{\mathsf{F}(\cone(f),Y)} \leq \level{\cat{T}}{X}{Y} ,
\end{equation*}
since $\mathsf{F}(\cone(f),-)$ is an exact functor. Hence the desired inequality holds by induction on $c$. 
\end{proof}

\begin{theorem} \label{kos_obj_level}
Let $\cat{T}$ be a triangulated category and $\bmalpha = \alpha_1, \ldots, \alpha_c$ a sequence of elements in $\ctr^*(\cat{T})$ each induced by an action of a monoidal triangulated category $\mathsf{F}_i \colon \cat{S}_i \times \cat{T} \to \cat{T}$. If each $\mathsf{F}_i$ admits a strong Verdier structure, then 
\begin{equation*}
\level{\cat{T}}{X}{\kosobj{Y}{(\alpha_1, \ldots, \alpha_c)}} \leq \level{\cat{T}}{X}{Y} + c
\end{equation*}
for any $X,Y \in \cat{T}$.
\end{theorem}
\begin{proof} By induction, it suffices to establish the inequality 
\begin{equation*}
\level{\cat{T}}{X}{\kosobj{Y}{\alpha}} \leq \level{\cat{T}}{X}{Y} + 1
\end{equation*}
for just one element $\alpha=\alpha_f$ induced by an action of a monoidal triangulated category $\mathsf{F} \colon \cat{S} \times \cat{T} \to \cat{T}$ admitting a strong Verdier structure.

Since $X\cong \mathsf{F}(\unit,X)$ and $\kosobj{Y}{\alpha}=\mathsf{F}(\cone(f),Y)$, the desired inequality follows from \cref{biexact_level} using that $\level{\cat{S}}{\unit}{\cone(f)} \leq 2$. 
\end{proof}

\begin{example} 
Let $R$ be a commutative ring. Then there is a natural embedding $R \to \ctr^*(\dcat{\Mod{R}})$. For $x_1, \ldots, x_c \in R$ we denote by $\Kos^R(x_1, \ldots, x_c)$ the Koszul complex on $x_1, \ldots, x_c$. Then
\begin{equation*}
\kosobj{Y}{(x_1, \ldots, x_c)} = Y \lotimes_R \Kos^R(x_1, \ldots, x_c)
\end{equation*}
for any $Y \in \dcat{\Mod{R}}$. The inequality in \cref{kos_obj_level} yields 
\begin{equation*}
\level{}{X}{Y \lotimes_R \Kos^R(x_1, \ldots, x_c)} \leq \level{}{X}{Y} + c .
\end{equation*}
When $Y$ is an $R$-module and $x_1, \ldots, x_c$ is a $Y$-regular sequence, then 
\begin{equation*}
Y \lotimes_R \Kos^R(x_1, \ldots, x_c) \simeq Y/(x_1, \ldots, x_c) .
\end{equation*}
Thus we recover the inequality
\begin{equation*}
\projdim_R(Y/(x_1, \ldots, x_c)) \leq \projdim_R(Y) + c
\end{equation*}
for $X = R$; see for example \cite[Exercise~1.3.6]{Bruns/Herzog:1998}. 
\end{example}

\begin{remark}
The upper bound in \cref{kos_obj_level} can be achieved. In fact, \cite[Theorem~3.3]{Bergh/Iyengar/Krause/Oppermann:2010} provides conditions such that
\begin{equation*}
\level{\cat{T}}{X}{\kosobj{X}{(\alpha_1, \ldots, \alpha_c)}} \geq c + 1
\end{equation*}
for some sequence $\alpha_1, \ldots, \alpha_c$.
\end{remark}

\section{Stable cofibration categories and their homotopy categories} \label{sec:cofib}

We are interested in triangulated categories that arise as homotopy categories of stable cofibration categories. {Such triangulated categories are called topological and they encompass all algebraic triangulated categories.} Similar to Quillen model categories, cofibration categories are a model for homotopy theory. They are dual to Brown's notion of categories of fibrant objects in \cite[I.1.]{Brown:1973} and correspond to precofibration categories in which all objects are cofibrant in the terminology of \cite{RadulescuBanu:2009}. Szumi{\l}o proved in \cite{Szumilo:2017} that the homotopy theory of cofibration categories is equivalent to the homotopy theory of finitely cocomplete $\infty$-categories. In this section we discuss definitions and important properties of cofibration categories and stable cofibration categories; for a general reference see \cite{Schwede:2013}. Schwede proved that the homotopy category of a stable cofibration category is triangulated. We will supplement his proof in \cref{triangulationstrong} to show that the triangulation is strong in the sense of May \cite{May:2001}. \cref{mayervietoris} extends \cite[Lemma~5.7]{May:2001} to relate pushouts in a stable cofibration category to homotopy pushouts in the associated triangulated category. 

\begin{discussion}
Any \emph{cofibration category} $\cat{C}$ comes with two classes of morphisms; a class of \emph{cofibrations} and a class of \emph{weak equivalences}. These are subject to the following axioms:
\begin{enumerate}
\item Every isomorphism is a weak equivalence and weak equivalences satisfy the $2$-out-of-$3$ property: For composable morphisms $f$ and $g$, if two out of $f$, $g$, $gf$ are weak equivalences, then so is the third.
\item Every isomorphism is a cofibration and cofibrations are closed under composition. 
\item Any diagram $Y \leftarrow X\rightarrow Z$ in which $Y\leftarrow X$ is a cofibration has a pushout $Y\cup_X Z$ and the morphism $Z \to Y \cup_X Z$ is a cofibration. If additionally, $Y\leftarrow X$ is a weak equivalence then so is $Z\to Y\cup_X Z$.
\item The category $\cat{C}$ has an initial object and every morphism from an initial object is a cofibration.
\item Any morphism $X\to Y$ in $\cat{C}$ can be factored as a cofibration followed by a weak equivalence.
\end{enumerate}
We write $X \rightarrowtail Y$ for a cofibration and $X \xrightarrow{\sim} Y$ for a weak equivalence. Maps that are both cofibrations and weak equivalences are called \emph{acyclic cofibrations}. We denote the localization of $\cat{C}$ at the weak equivalences by $\gamma\colon \cat{C} \to \Ho(\cat{C})$; see \cite[I.1.]{Gabriel/Zisman:1967}. Given a zig-zag
\[
Y_1 \xleftarrow{s_1} X_1 \xrightarrow{f_1} Y_2 \xleftarrow{s_2} X_2\rightarrow \ldots\leftarrow X_n \xrightarrow{f_n} Y_n
\]
in $\cat{C}$, where each $s_i$ is a weak equivalence, we write
\[
\gamma(Y_1 \xleftarrow{\sim} X_1 \rightarrow Y_2 \xleftarrow{\sim} X_2\rightarrow \ldots \xleftarrow{\sim} X_n \rightarrow Y_n)
\]
for the composite $\gamma(f_n)\ldots\gamma(s_2)^{-1}\gamma(f_1)\gamma(s_1)^{-1}$ in $\Ho(\cat{C})$. 
\end{discussion}

\begin{remark}
The homotopy category $\Ho(\cat{C})$ of a cofibration category may not be locally small; see \cite[Remark~A.2]{Schwede:2013}. For any model category $\cat{M}$, the full subcategory of cofibrant objects together with the cofibrations and weak equivalences of $\cat{M}$ between cofibrant objects is a cofibration category. This is the main family of examples of interest for us. In this case, the homotopy category is locally small, since the homotopy category of cofibrant objects in $\cat{M}$ is equivalent to the homotopy category of $\cat{M}$ and the latter is locally small.
\end{remark}

Let $\cat{C}$ be a pointed cofibration category. Pointed means that every initial object is also a terminal object. We write $\ast$ for the terminal object. When $X \to Y$ is a cofibration we write $Y/X$ for the pushout $Y \cup_X \ast$. 

\begin{discussion} \label{cofib_cat_cone}
An object $C$ is \emph{weakly contractible} if $C \to \ast$ is a weak equivalence. A cofibration $X \to C$ with $C$ weakly contractible is called a \emph{cone (of $X$)}. 

A morphism between objects extends to a morphism between cones in the following way: For any morphism $f \colon X\to Y$ and cones $X \to C_X$ and $Y \to C_Y$ in $\cat{C}$, there exists a morphism $\bar{f} \colon C_X \to \bar{C}$ and an acyclic cofibration $s \colon C_Y \to \bar{C}$ such that 
\[
\begin{tikzcd}
X \ar[rr,"f"] \ar[d,tail] & & Y \ar[d,tail] \\
C_X \ar[r,"\bar{f}" swap] & \bar{C} & C_Y \ar[l,"s",tail,"\sim" swap] 
\end{tikzcd}
\]
commutes in $\cat{C}$ and $C_X \cup_X C_Y \to \bar{C}$ is a cofibration; see \cite[Lemma~A.3]{Schwede:2013}. The pair $(\bar{f},s)$ is called a \emph{cone extension} of $f\colon X \to Y$. Instead of $(\bar{f},s)$ we often say $\bar{C}$ is the cone extension. Note, that if $f$ is a cofibration, then so is $\bar{f}$. The composite 
\[
\gamma(C_X/X \xrightarrow{\bar{f}/f} \bar{C}/Y \xleftarrow{s/\id_Y} C_Y/Y)
\]
in the homotopy category is independent of the chosen cone extension.
\end{discussion}

\begin{discussion} \label{cofib_cat_triangulated}
In a pointed cofibration category $\cat{C}$ we fix a cone $X \to CX$ for every object $X$. The \emph{suspension} of $X$ is $\susp X \coloneqq CX/X$. By \cite[Proposition~A.4]{Schwede:2013} this defines a functor $\susp\colon \cat{C} \to \Ho(\cat{C})$ which is given on morphisms as
\begin{equation*}
\susp (X \xrightarrow{f} Y) \coloneqq \gamma(CX/X \to \bar{C}/Y \gets CY/Y)
\end{equation*}
where $\bar{C}$ is a cone extension of $f$. Moreover, the suspension functor $\susp$ takes weak equivalences to isomorphisms and thus induces a functor $\susp\colon \Ho(\cat{C})\to \Ho(\cat{C})$. 
\end{discussion}

\begin{discussion} \label{connectmorphnat}
Let $f \colon X \rightarrowtail Y$ be a cofibration in $\cat{C}$. The \emph{connecting morphism} of $f$ is
\begin{equation*}
\delta(f) \coloneqq \gamma(Y/X \xleftarrow{\sim} CX \cup_X Y \to \susp X) .
\end{equation*}
We call
\begin{equation*}
X \xrightarrow{\gamma(f)} Y \to Y/X \xrightarrow{\delta(f)} \susp X
\end{equation*}
an \emph{elementary exact triangle}. 

We will use repeatedly that the connecting morphism is natural in the sense of \cite[Proposition~A.11]{Schwede:2013}: For any commutative square
\[
\begin{tikzcd}
X\ar[r,tail,"f"]\ar[d] & Y\ar[d] \\
X'\ar[r,tail,"f'"] & Y'
\end{tikzcd}
\]
in $\mathcal{C}$, where $f$ and $f'$ are cofibrations, the diagram
\[
\begin{tikzcd}
Y/X\ar[r,"\delta(f)"]\ar[d] & \susp X\ar[d] \\
Y'/X'\ar[r,"\delta(f')"] & \susp X'
\end{tikzcd}
\]
commutes in $\Ho(\cat{C})$.
\end{discussion}

\subsection{Stable cofibration categories}

If the induced endofunctor $\susp$ on $\Ho(\cat{C})$ is an equivalence, then the pointed cofibration category $\cat{C}$ is said to be \emph{stable}. Schwede proved in \cite[Theorem~A.12]{Schwede:2013} that the homotopy category of any stable cofibration category is triangulated where the exact triangles are those isomorphic to an elementary exact triangle.

May \cite[Lemma~5.7]{May:2001} proved for some families of stable model categories that a pushout square of cofibrant objects, where two parallel morphisms are cofibrations, induces a homotopy cartesian square in the homotopy category. Groth, Ponto, and Shulman established the result for stable derivators in \cite[Theorem 6.1]{Groth/Ponto/Shulman:2014a}, in particular extending May's result to all stable model categories. Combining May's arguments with Schwede's results we provide a proof of the analogous result for stable cofibration categories so that there is no need to change frameworks.

\begin{lemma}\label{mayervietoris}
Let
\[
\begin{tikzcd}
X \ar[r,tail,"f"] \ar[d,"g" swap] & Y \ar[d,"f'"] \\
Z \ar[r,tail,"g'" swap] & P
\end{tikzcd}
\]
be a pushout square in a stable cofibration category $\cat{C}$ with $f\colon X\to Y$ a cofibration. Then its image in $\Ho(\cat{C})$ is a homotopy cartesian square.
\end{lemma}
\begin{proof}
By rotation, it is enough to construct an exact triangle
\begin{equation*}
Y \oplus Z \xrightarrow{\left(\begin{smallmatrix}g' & f'\end{smallmatrix}\right)} P \xrightarrow{\partial} \susp X \xrightarrow{\susp\left(\begin{smallmatrix} -f \\ g\end{smallmatrix}\right)} \susp(Y\oplus Z) .
\end{equation*}
Let $\cyl(X)$ be a cylinder object for $X$; that is we factor the fold map $X\sqcup X\to X$ into a cofibration followed by a weak equivalence
\begin{equation*}
\begin{tikzcd}
X\sqcup X\ar[r,tail,"{(i_0,i_1)}"] & \cyl(X)\ar[r,"q","\sim" swap] & X ;
\end{tikzcd}
\end{equation*}
see for example \cite[Section~1.5]{RadulescuBanu:2009}. Let 
\begin{equation*}
M(f,g)\coloneqq\cyl(X)\cup_{X\sqcup X} (Y\sqcup Z)= Y\cup_X \cyl(X) \cup_X Z
\end{equation*}
be the double mapping cylinder of $(f,g)$. The gluing lemma applied to 
\[
\begin{tikzcd}
\cyl(X)\cup_X Z\ar[d,"\sim"] & \cyl(X) \ar[l] & X \ar[l,"i_0" swap] \ar[r,tail,"f"] \ar[d,"="] & Y\ar[d,"="] \\
Z && X \ar[ll,"g" swap] \ar[r,tail,"f"] & Y
\end{tikzcd}
\]
provides a weak equivalence $M(f,g) \xrightarrow{\sim} P$; see \cite[Lemma~1.4.1(1)(b)]{RadulescuBanu:2009}. Moreover, 
\begin{equation*}
M(f,g)/(Y\sqcup Z)\cong \cyl(X)/(X\sqcup X)
\end{equation*}
and we have weak equivalences
\[
\cyl(X)/(X\sqcup X)=\ast\cup_X \cyl(X)\cup_X \ast \xleftarrow{\sim} \ast \cup_X \cyl(X) \cup_X CX \xrightarrow{\sim} \ast \cup_X CX = \susp X.
\] 

It follows that the elementary exact triangle arising from the cofibration $Y\sqcup Z\rightarrowtail M(f,g)$ yields an exact triangle 
\[
Y \oplus Z \xrightarrow{\left(\begin{smallmatrix}g' & f'\end{smallmatrix}\right)} P \rightarrow \susp X \xrightarrow{\delta} \susp(Y\oplus Z) .
\]
It remains to check that $\delta$ is the suspension of $(-f,g)^T \colon X\to Y\oplus Z$. We show that its projection to $\susp Y$ is $-\susp f$. Choosing a cone extension $CX\to C\gets CZ$ of $g\colon X\to Z$, this follows from the naturality of the connecting morphism \cref{connectmorphnat} applied to 
\[
\begin{tikzcd}
Y\sqcup Z\ar[r,tail]\ar[d] & Y\cup_X \cyl(X)\cup_X Z\ar[d] \\
Y\sqcup C \ar[r,tail] & Y\cup_X \cyl(X)\cup_X C \\
Y\ar[r,tail]\ar[u]\ar[d] & Y\cup_X \cyl(X)\cup_X CX\ar[u]\ar[d, "\id_Y \cup q \cup_X \id_{CX}"]\\
Y\ar[r,tail] & Y\cup_X X \cup_X CX \nospacepunct{.}
\end{tikzcd}
\]
The connecting morphism of the cofibration $Y\to Y\cup_X CX$ is indeed $-\susp f$ by the proof of the rotation axiom in \cite[Theorem~A.12]{Schwede:2013}.

Similarly, the projection of $\delta \colon \susp X\to \susp(Y\oplus Z)$ to $\susp Z$ is $\susp g$. The sign of $-\susp g$ cancels with the sign arising from
\[
\gamma(\susp X = \ast \cup_X CX \xleftarrow{\sim} CX\cup_X CX \xrightarrow{\sim} CX\cup_X\ast = \susp X) = -\id_{\susp X} ;
\]
see \cite[Proposition~A.8(iii)]{Schwede:2013}.
\end{proof}

The following supplements the proof of the octahedral axiom in $\Ho(\cat{C})$; see \cite[Theorem~A.12]{Schwede:2013}. In particular, we show that the triangulation of $\Ho(\cat{C})$ is \emph{strong} in the sense of \cite[Definition~3.8]{May:2001}.

\begin{proposition}\label{triangulationstrong}
Let $\begin{tikzcd}[column sep=small] X \ar[r,tail,"f"] & Y \ar[r,tail,"g"] & Z \end{tikzcd}$ be a composition of cofibrations in a stable cofibration category $\cat{C}$. Then
\begin{equation*}
\begin{tikzcd}
Y \ar[r] \ar[d,"\gamma(g)" swap] & Y/X \ar[d] \\
Z \ar[r] & Z/X
\end{tikzcd}
\quad \text{and} \quad
\begin{tikzcd}
Z/X \ar[r,"\delta(gf)"] \ar[d] & \susp X \ar[d,"\susp f"] \\
Z/Y \ar[r,"\delta(g)"] & \susp Y
\end{tikzcd}
\end{equation*}
are homotopy pushout squares in $\Ho(\cat{C})$.
\end{proposition}
\begin{proof} 
The first square considered in $\mathcal{C}$ fits into a commutative diagram
\begin{equation*}
\begin{tikzcd}
X \ar[d] \ar[r,tail] & Y \ar[d] \ar[r,tail] & Z \ar[d] \\
\ast \ar[r,tail] & Y/X \ar[r,tail] & Z/X
\end{tikzcd}
\end{equation*}
and it follows from the pasting lemma for pushouts that it is a pushout square. Thus its image is a homotopy pushout square in $\Ho(\cat{C})$ by \cref{mayervietoris}. 

We show that the second square is homotopy cartesian. We pick a cone extension $\bar{C}$ of $f$. This fits into a diagram
\begin{equation*}
\begin{tikzcd}
Z/X \ar[d] & CX \cup_X Z \ar[r] \ar[d] \ar[l,"\sim"] & CX/X \ar[d] \\
Z/Y & \bar{C}\cup_Y Z \ar[l,"\sim"]\ar[r] & \bar{C}/Y \\
Z/Y \ar[u,"=" swap] & CY \cup_Y Z \ar[l,"\sim"] \ar[u,"\sim" swap] \ar[r] & CY/Y \ar[u,"\sim" swap]
\end{tikzcd}
\end{equation*}
consisting of four commutative squares in $\cat{C}$. Since weak equivalences become isomorphisms in $\Ho(\cat{C})$, it is enough to show that the upper right-hand square is homotopy cartesian in $\Ho(\cat{C})$. We will show that it is a pushout square in $\cat{C}$ and that $CX \cup_X Z \to \bar{C} \cup_Y Z$ is a cofibration in order to apply \cref{mayervietoris}.

The pasting lemma for pushouts yields diagrams of pushout squares
\begin{equation*}
\begin{tikzcd}
X \ar[r,tail] \ar[d,tail] & Z \ar[r] \ar[d,tail] & \ast \ar[d,tail] \\
CX \ar[r,tail] & CX \cup_X Z \ar[r] & \susp X
\end{tikzcd}
\quad \text{and} \quad
\begin{tikzcd}
Y \ar[r,tail] \ar[dd,tail] & Z \ar[r] \ar[d,tail] & \ast \ar[d,tail] \\
& CX \cup_X Z \ar[r] \ar[d] & \susp X \ar[d] \\
\bar{C} \ar[r,tail] & \bar{C} \cup_Y Z \ar[r] & \bar{C}/Y \nospacepunct{.}
\end{tikzcd}
\end{equation*}
By \cref{cofib_cat_cone} the morphism $CX \cup_X CY \to \bar{C}$ is a cofibration, and hence so is $CX \cup_X Y \to \bar{C}$. We apply \cite[Lemma~1.4.1(1)(a)]{RadulescuBanu:2009} to
\begin{equation*}
\begin{tikzcd}
CX \ar[d,tail] & X \ar[l,tail] \ar[d,tail] \ar[r,tail] & Z \ar[d,"="] \\
\bar{C} & Y \ar[l,tail] \ar[r,tail] & Z
\end{tikzcd}
\end{equation*}
and obtain that $CX \cup_X Z \to \bar{C} \cup_Y Z$ is a cofibration. Thus \cref{mayervietoris} applied to the square with corners $CX \cup_X Z$, $\bar{C} \cup_Y Z$, $\susp X$, $\bar{C}/Y$ provides the desired homotopy cartesian square.
\end{proof}

\begin{corollary}
The homotopy category of a stable cofibration category is a strongly triangulated category in the sense of \cite[Definition~3.8]{May:2001}.
\end{corollary}
\begin{proof}
By the same argument as in the proof of the octahedral axiom in \cite[Theorem~A.12]{Schwede:2013} it is enough to consider a composition $X \xrightarrow{f} Y \xrightarrow{g} Z$ in the homotopy category where $f = \gamma(f')$ and $g = \gamma(g')$ for cofibrations $f'$ and $g'$. Then the claim holds by \cref{triangulationstrong}.
\end{proof}

\section{Biexact functors between stable cofibration categories}
\label{sec:cofib_biexact}

In this section we show that a biexact functor between stable cofibration categories induces a biexact functor of triangulated categories and that the induced functor admits a strong Verdier structure; see \cref{cofib_biexact_verdier}. Analogous results have been proved for monoidal products; for some families of monoidal stable model categories see May \cite[Section~6]{May:2001} and for strong, stable monoidal derivators see \cite[Theorem~6.2]{Groth/Ponto/Shulman:2014}. In particular, we show that May's arguments extend to any Quillen bifunctor between stable model categories as the bifunctors can be restricted to biexact functors between categories of cofibrant objects.

\begin{discussion} \label{exact_fun_cofib}
A functor $\mathsf{F} \colon \cat{C}\to \cat{D}$ between cofibration categories is \emph{exact} if it preserves initial objects, cofibrations, weak equivalences, and pushouts of diagrams, where one leg is a cofibration. Since $\mathsf{F}$ preserves weak equivalences, it induces a functor $\Ho(\mathsf{F}) \colon \Ho(\cat{C})\to \Ho(\cat{D})$ on the homotopy categories. If $\cat{C}$ and $\cat{D}$ are stable, then $\Ho(\mathsf{F})$ is an exact functor of triangulated categories by \cite[Proposition~A.14]{Schwede:2013}. 

The natural isomorphism $\tau \colon \Ho(\mathsf{F}) \susp \to \susp \Ho(\mathsf{F})$ is given as follows: We fix an object $X$ and let $C$ be a cone extension of the identity on $\mathsf{F}(X)$ with respect to the cones $C \mathsf{F}(X)$ and $\mathsf{F}(CX)$. Then
\begin{equation*}
\tau(X) = \gamma(\mathsf{F}(CX)/\mathsf{F}(X) \to C/\mathsf{F}(X) \xleftarrow{\sim} C \mathsf{F}(X)/\mathsf{F}(X)) .
\end{equation*}
\end{discussion}

\begin{lemma} \label{exact_nat_susp}
Let $\mathsf{F},\mathsf{G} \colon \cat{C} \to \cat{D}$ be exact functors between stable cofibration categories and $\eta \colon \mathsf{F} \to \mathsf{G}$ a natural transformation. Then $\Ho(\eta) \colon \Ho(\mathsf{F}) \to \Ho(\mathsf{G})$ is a natural transformation of exact functors; that is $(\susp \Ho(\eta)) \tau^\mathsf{F} = \tau^\mathsf{G} (\Ho(\eta) \susp)$. 
\end{lemma}
\begin{proof}
We pick cone extensions $C_\mathsf{F}$ and $C_\mathsf{G}$ of the identity on $\mathsf{F}(X)$ and $\mathsf{G}(X)$, respectively, as above. Then $\mathsf{F}(CX) \to C_\mathsf{F}$ and $\mathsf{G}(CX) \to C_\mathsf{G}$ are cones, and we pick a cone extension $C$ of $\eta(CX)$ with respect to these cones. Then we obtain a commutative diagram
\begin{equation*}
\begin{tikzcd}
\mathsf{F}(CX)/\mathsf{F}(X) \ar[r,"\sim" swap] \ar[dd,"\eta" swap] & C_\mathsf{F}/\mathsf{F}(X) \ar[d] & C \mathsf{F}(X)/\mathsf{F}(X) \ar[l,"\sim"] \ar[dl] \\
& C/\mathsf{G}(X) \\
\mathsf{G}(CX)/\mathsf{G}(X) \ar[ur,"\sim" swap] \ar[r,"\sim" swap] & C_\mathsf{G}/\mathsf{G}(X) \ar[u,"\sim" swap] & C \mathsf{G}(X)/\mathsf{G}(X) \ar[l,"\sim"] \ar[ul,"\sim"] \nospacepunct{.}
\end{tikzcd}
\end{equation*}
It remains to observe, that in $\Ho(\cat{D})$ the morphisms in the top and bottom row are $\tau^\mathsf{F}$ and $\tau^\mathsf{G}$, respectively, the morphism in the left column is $\Ho(\eta)(\susp X)$ and the morphism in the right column is $\susp \Ho(\eta)(X)$.
\end{proof}

\begin{definition} \label{biexact_cofib}
Let $\cat{C}$, $\cat{D}$, $\cat{E}$ be cofibration categories. A functor $\mathsf{F}\colon \cat{C}\times\cat{D}\to \cat{E}$ is \emph{biexact}, if it is exact in each variable and for any morphisms $f\colon X\to Y$ in $\cat{C}$ and $f' \colon X'\to Y'$ in $\cat{D}$, the induced morphism 
\begin{equation} \label{biexact_cofib:compatible}
\mathsf{F}(X,Y')\cup_{\mathsf{F}(X,X')}\mathsf{F}(Y,X')\to \mathsf{F}(Y,Y')
\end{equation}
is a cofibration provided that $f$ and $f'$ are cofibrations.
\end{definition}

A biexact functor $\mathsf{F}\colon \cat{C}\times\cat{D}\to \cat{E}$ preserves weak equivalences and thus induces a functor $\Ho(\mathsf{F})\colon \Ho(\cat{C})\times \Ho(\cat{D})\to \Ho(\cat{E})$ which coincides with $\mathsf{F}$ on objects.

\begin{theorem}\label{cofib_biexact_verdier}
Let $\mathsf{F}\colon \cat{C}\times\cat{D}\to \cat{E}$ be a biexact functor between stable cofibration categories. Then the induced functor $\Ho(\mathsf{F})$ is a biexact functor of triangulated categories and $\Ho(\mathsf{F})$ admits a strong Verdier structure.
\end{theorem} 
\begin{proof}
By \cref{exact_fun_cofib,exact_nat_susp}, we have natural isomorphisms
\begin{equation*}
\theta \colon \Ho(\mathsf{F})(\susp -,-) \to \susp \Ho(\mathsf{F})(-,-) \quad \text{and} \quad \zeta \colon \Ho(\mathsf{F})(-,\susp -) \to \susp \Ho(\mathsf{F})(-,-) .
\end{equation*}
We will show that $\Ho(\mathsf{F})$ admits a strong Verdier structure. This will imply the compatibility condition \cref{biexact:anticommute}. 

It suffices to establish the strong Verdier structure for elementary exact triangles
\[
X \xrightarrow{\gamma(f)} Y \xrightarrow{\gamma(g)} Z \xrightarrow{\delta(f)} \susp X \quad \text{and} \quad X' \xrightarrow{\gamma(f')} Y' \xrightarrow{\gamma(g')} Z' \xrightarrow{\delta(f')} \susp X'
\]
in $\Ho(\cat{C})$ and $\Ho(\cat{D})$, respectively. We make constructions in the stable cofibration category and, for readability, we use the labeling of \cref{eq:bifunctorVerdier} to refer to the corresponding squares in the cofibration category.

We set $W \coloneqq \mathsf{F}(X,Y')\cup_{\mathsf{F}(X,X')} \mathsf{F}(Y,X')$. By \cref{biexact_cofib:compatible} the morphism $W \to \mathsf{F}(Y,Y')$ is a cofibration. Using the pasting lemma repeatedly, we obtain the following diagram in which every square is a pushout square:
\begin{equation} \label{cofib_biexact_verdier:W}
\begin{tikzcd}
\mathsf{F}(X,X') \ar[r,tail] \ar[d,tail] \ar[dr,phantom,"\textrm{(I)}"] & \mathsf{F}(Y,X') \ar[d,tail] \ar[r] & \ast \ar[d,tail] \\
\mathsf{F}(X,Y') \ar[r,tail] & W \ar[r] \ar[d,tail] \ar[dr,phantom,"\textrm{(III)}"] & \mathsf{F}(X,Z') \ar[r] \ar[d,tail] & \ast \ar[d,tail] \\
& \mathsf{F}(Y,Y') \ar[r] & \mathsf{F}(Y,Z') \ar[r] & \mathsf{F}(Z,Z') \nospacepunct{.}
\end{tikzcd}
\end{equation}
This yields the elementary exact triangles 
\[
\begin{tikzcd}[row sep=0]
W \ar[r,"i"] & \mathsf{F}(Y,Y') \ar[r,dotted] & \mathsf{F}(Z,Z') \ar[r,"{p = \delta(i)}"] & \susp W \\
\mathsf{F}(Y,X') \ar[r,"j'"] & W \ar[r,"q'"] & \mathsf{F}(X,Z') \ar[r,dotted] & \susp \mathsf{F}(Y,X') \nospacepunct{,}
\end{tikzcd}
\]
and similarly, we obtain (II) and an elementary exact triangle
\[
\begin{tikzcd}
\mathsf{F}(X,Y') \ar[r,"j"] & W \ar[r,"q"] & \mathsf{F}(Z,X') \ar[r,dotted] & \susp \mathsf{F}(X,Y') \nospacepunct{.}
\end{tikzcd}
\]
The dotted arrow in the first triangle is $\gamma(\mathsf{F}(g,g'))$. In the latter two triangles the dotted arrows are $(\susp \gamma(\mathsf{F}(f,X'))) \delta(\mathsf{F}(X,f'))$ and $(\susp \gamma(\mathsf{F}(X,f'))) \delta(\mathsf{F}(f,X'))$, respectively, by applying the naturality of the connecting morphism \cref{connectmorphnat} to the defining pushout square of $W$. Since $\delta(\mathsf{F}(X,f')) = \zeta \Ho(\mathsf{F})(X,\delta(f'))$, and similarly for $\delta(\mathsf{F}(f,X'))$, these are the required triangles \cref{biexact:exact-triangles}.

We need to check, that the morphisms satisfy the compatibility conditions in \cref{eq:bifunctorVerdier}, and that the required squares are homotopy cartesian. 

By construction (i--iv) commute. The triangles (v) and (vi) commute by the naturality of the connecting morphism in \cref{connectmorphnat}. Further by \cref{cofib_biexact_verdier:W,mayervietoris} the squares (I--III) commute and are homotopy pushout squares. 

Next we show that (VI) is anti-commutative and homotopy cartesian when replacing one arrow by its negative. Pasting pushouts in the cube
\setlength{\perspective}{9pt}
\begin{equation*}
\begin{tikzcd}[row sep=1.5em, column sep=1.5em]
&[-\perspective] \mathsf{F}(Y,X') \ar{rr}\ar{dd} &[\perspective] &[-\perspective] W \ar{dd} \\[-\perspective]
\mathsf{F}(X,X') \ar[crossing over]{rr}\ar{dd}\ar{ru} & & \mathsf{F}(X,Y') \ar{ru} \\[\perspective]
& \mathsf{F}(Z,X') \ar{rr} & & \mathsf{F}(Z,X') \sqcup \mathsf{F}(X,Z') \\[-\perspective]
\ast \ar{rr}\ar{ru} && \mathsf{F}(X,Z') \ar[from=uu, crossing over]\ar{ru}
\end{tikzcd}
\end{equation*}
shows that $W/\mathsf{F}(X,X')\cong \mathsf{F}(Z,X')\sqcup \mathsf{F}(X,Z')$. By construction the composition 
\begin{equation*}
W \to \mathsf{F}(Z,X') \sqcup \mathsf{F}(X,Z') \to \mathsf{F}(Z,X') \sqcup \ast \cong \mathsf{F}(Z,X')
\end{equation*}
coincides with the morphism $W \to \mathsf{F}(Z,X')$ above; and similarly for the second summand. Thus we have an elementary exact triangle 
\begin{equation*}
\mathsf{F}(X,X')\to W\to \mathsf{F}(X,Z') \oplus \mathsf{F}(Z,X') \xrightarrow{\delta} \susp \mathsf{F}(X,X')
\end{equation*}
in $\Ho(\cat{E})$. It follows from \cref{connectmorphnat} that the connecting morphism $\delta$ restricts to $\delta(\mathsf{F}(X,f'))$ and $\delta(\mathsf{F}(f,X))$ on the summands $\mathsf{F}(X,Z')$ and $\mathsf{F}(Z,X')$, respectively. Rotating this triangle shows that (VI) anti-commutes, and it is homotopy cartesian if we replace one arrow by its negative. 

Applying \cref{triangulationstrong} to the compositions of cofibrations $\mathsf{F}(X,Y') \to W \to \mathsf{F}(Y,Y')$ and $\mathsf{F}(Y,X') \to W \to \mathsf{F}(Y,Y')$ shows that (IV) and (V) are commutative and homotopy cartesian, respectively. 

It remains to observe, that \cref{biexact:anticommute} is a consequence of \cref{biexact:exact-triangles} by using the exact triangles 
\begin{equation*}
X \to 0 \to \susp X \xrightarrow{-\id_{\susp X}} \susp X \quad \text{and} \quad X' \to 0 \to \susp X' \xrightarrow{-\id_{\susp X'}} \susp X' .
\end{equation*}
Hence $\Ho(\mathsf{F})$ is a biexact functor that admits a strong Verdier structure.
\end{proof}

\begin{remark}
Let $X \to Y \to Z \to \susp X$ and $X' \to Y' \to Z' \to \susp X'$ be exact triangles in $\Ho(\cat{C})$, and $W$ the object constructed from these triangles as in the proof of \cref{cofib_biexact_verdier}. When we replace the triangle $X \to Y \to Z \to \susp X$ by $Y \to Z \to \susp X \to \susp Y$ (with appropriate signs) in the proof of \cref{cofib_biexact_verdier}, we obtain an object $V$. Then there exist homotopy pushout squares
\[
\begin{tikzcd}
V \ar[r] \ar[d] & \susp \mathsf{F}(Y,X') \ar[d] \\
\mathsf{F}(Z,Y') \ar[r] & \susp W
\end{tikzcd}
\quad \text{and} \quad
\begin{tikzcd}
W \ar[r] \ar[d] & \mathsf{F}(Y,Y') \ar[d] \\
\mathsf{F}(X,Z') \oplus \mathsf{F}(Z,X') \ar[r] & V \nospacepunct{.}
\end{tikzcd}
\]
in $\Ho(\cat{C})$. By \cite[Theorem~4.1]{Keller/Neeman:2002}, this is also a direct consequence of the strong Verdier structure and the octahedral axiom. 

The second homotopy pushout square also appears in \cite{May:2001} as (TC4). Using \cref{verdier_rotation} below, it is straightforward to check that instead of replacing $X \to Y \to Z \to \susp X$ by $Y \to Z \to \susp X \to \susp Y$, one can obtain $\susp^{-1} V$ from $\susp^{-1} Z \to X \to Y \to Z$ and $\susp^{-1} Z' \to X' \to Y' \to Z'$.
\end{remark}

\subsection{Naturality}

In the proof of \cref{cofib_biexact_verdier} we constructed the object $W$ as a pushout in the stable cofibration category. We will show that the strong Verdier structure on the bifunctor on the homotopy categories inherits some of the naturality from the bifunctor on the stable cofibration categories. These compatibility results may be of interest for future applications, but we do not use them here.

\begin{discussion}\label{lifting_squares}
When we work in the homotopy category $\Ho(\cat{C})$ of a cofibration category $\cat{C}$, it would be helpful to lift commutative diagrams in $\Ho(\cat{C})$ to commutative diagrams in $\cat{C}$ to use the structure of $\cat{C}$. This works for commutative squares: As in the proof of (T4) in \cite[Theorem~A.12]{Schwede:2013}, any morphism in $\Ho(\cat{C})$ factors as the localization functor $\gamma\colon \cat{C} \to \Ho(\cat{C})$ applied to a cofibration followed by an isomorphism. Thus any commutative square in $\Ho(\cat{C})$ is isomorphic to a commutative square such that two parallel arrows are $\gamma$ of cofibrations. Schwede then shows in the proof of (T3) in \cite[Theorem~A.12]{Schwede:2013}, that any such commutative square in $\Ho(\cat{C})$ is isomorphic to $\gamma$ applied to a commutative square in $\cat{C}$ in which two parallel arrows are cofibrations. 
\end{discussion}

\begin{lemma} \label{verdier_functorial}
Let $\mathsf{F} \colon \cat{C} \times \cat{D} \to \cat{E}$ be a biexact functor of stable cofibration categories and let
\begin{equation*}
\begin{tikzcd}
X_1 \ar[r] \ar[d] & Y_1 \ar[d] \\
X_2 \ar[r] & Y_2
\end{tikzcd}
\quad \text{and} \quad
\begin{tikzcd}
X'_1 \ar[r] \ar[d] & Y'_1 \ar[d] \\
X'_2 \ar[r] & Y'_2
\end{tikzcd}
\end{equation*}
be commutative squares in $\Ho(\cat{C})$ and $\Ho(\cat{D})$, respectively. Then there exist homotopy pushouts $W_i$ of $\mathsf{F}(X_i,Y'_i) \leftarrow \mathsf{F}(X_i,X'_i) \rightarrow \mathsf{F}(Y_i,X'_i)$ in $\Ho(\cat{E})$ satisfying \cref{verdier_structure} and a morphism $W_1 \to W_2$ in $\Ho(\cat{E})$ that is compatible with the exact triangles \cref{biexact:exact-triangles}. 

Moreover, the cone of the morphism $W_1 \to W_2$ is isomorphic to
\begin{equation*}
\cone(\mathsf{F}(X_1,Y'_1) \to \mathsf{F}(X_2,Y'_2)) +_{\cone(\mathsf{F}(X_1,X'_1) \to \mathsf{F}(X_2,X'_2))} \cone(\mathsf{F}(Y_1,X'_1) \to \mathsf{F}(Y_2,X'_2)) .
\end{equation*}
\end{lemma}
\begin{proof}
By \cref{lifting_squares}, we can assume that the squares lift to commutative squares in $\cat{C}$ and $\cat{D}$, respectively, with the horizontal morphisms being cofibrations. 
Then we obtain the morphism $W_1 \to W_2$ from the morphism of the spans. The pushout property yields morphisms of the diagrams \cref{cofib_biexact_verdier:W} and thus a morphism of triangles in $\Ho(\cat{E})$. 

For the second claim, let $C_1$ be a cone extension of $\mathsf{F}(X_1,X'_1) \to \mathsf{F}(X_1,Y'_1)$ and $C_2$ a cone extension of $\mathsf{F}(X_1,X'_1) \to \mathsf{F}(Y_1,X'_1)$. We consider the commutative diagram
\begin{equation*}
\begin{tikzcd}
C_1 & C \mathsf{F}(X_1,X'_1) \ar[l,tail] \ar[r,tail] & C_2 \\
\mathsf{F}(X_1,Y'_1) \ar[u,tail] \ar[d] & \mathsf{F}(X_1,X'_1) \ar[l,tail] \ar[r,tail] \ar[u,tail] \ar[d] & \mathsf{F}(Y_1,X'_1) \ar[u,tail] \ar[d] \\
\mathsf{F}(X_2,Y'_2) & \mathsf{F}(X_2,X'_2) \ar[l,tail] \ar[r,tail] & \mathsf{F}(Y_2,X'_2) \nospacepunct{.}
\end{tikzcd}
\end{equation*}
Whether we first take pushouts horizontally and then vertically, or the other way around, we obtain the same object. By \cite[Lemma~1.4.1(1)(a)]{RadulescuBanu:2009} these pushouts exist. In $\Ho(\cat{E})$ the former yields $\cone(W_1 \to W_2)$, and the latter the desired homotopy pushout using \cref{mayervietoris}. 
\end{proof}

One should compare the previous \cref{verdier_functorial} to \cref{htpy_pushout_same_base,htpy_pushout_map_base}, as the latter treat special cases, though in a more general setting.

Next we show that the construction of the strong Verdier structure in \cref{cofib_biexact_verdier} is invariant under rotation. Namely, if we rotate the triangle in the first component of the biexact functor, then the three exact triangles \cref{biexact:exact-triangles} from the corresponding strong Verdier structure are up to rotation canonically isomorphic to the three exact triangles obtained from the strong Verdier structure after rotating the triangle in the second component.

We will use Schwede's proof of the rotation axiom from \cite[Theorem~A.12]{Schwede:2013}.

\begin{remark}\label{rotation_stable_cof_cat}
Let $f\colon X\to Y$ be a cofibration in a stable cofibration category $\cat{C}$. Then 
\[\begin{tikzcd}
X \ar[r]\ar[d,"="] & CX\cup_X Y \ar[r]\ar[d] & \susp X \ar[r,"\delta"] \ar[d,"="] & \susp Y\ar[d,"="] \\
X \ar[r] & Y/X \ar[r] & \susp X \ar[r," -\susp (f)"] & \susp Y \nospacepunct{.}
\end{tikzcd}
\]
is an isomorphism of exact triangles. If $C_X$ is another cone for $X$ and $\overline{C}$ a cone extension for $CX$ and $C_X$, then the latter exact triangle is isomorphic to
\[\begin{tikzcd}
X\ar[r] & \ar[r] Y/X\ar[r] & C_X/X\ar[r,"\delta"] & \susp Y,
\end{tikzcd}
\]
via $\gamma(CX/X\rightarrow \overline{C}/X\leftarrow C_X/X)$ and the identity morphisms on the other objects.
\end{remark}

\begin{lemma} \label{verdier_rotation}
Let $\mathsf{F} \colon \cat{C} \times \cat{D} \to \cat{E}$ be a biexact functor of stable cofibration categories and $f\colon X \to Y$ and $f'\colon X' \to Y'$ cofibrations in $\cat{C}$ and $\cat{D}$, respectively. Let $W_1$ be the pushout of $\mathsf{F}(X,CX' \cup_{X'} Y') \leftarrow \mathsf{F}(X,Y') \to \mathsf{F}(Y,Y')$ and $W_2$ the pushout of $\mathsf{F}(CX \cup_X Y,X') \leftarrow \mathsf{F}(Y,X') \to \mathsf{F}(Y,Y')$. Then there is a canonical isomorphism $e \colon W_1 \to W_2$ in $\Ho(\cat{E})$ that yields canonical isomorphisms of the exact triangles \cref{biexact:exact-triangles}; that is for $Z = Y/X$ and $Z'=Y'/X'$ the following diagrams commute
\begin{equation} \label{verdier_rotation:j}
\begin{tikzcd}[column sep=large]
\mathsf{F}(Y,Y') \ar[r,"j'_1"] \ar[d,"="] & W_1 \ar[r,"q'_1"] \ar[d,"e" swap,"\cong"] & \mathsf{F}(X,\susp X') \ar[r,"{\zeta \mathsf{F}(f,\susp f')}"] \ar[d,"\cong"] & \susp \mathsf{F}(Y,Y') \ar[d,"="] \\
\mathsf{F}(Y,Y') \ar[r,"j_2"] & W_2 \ar[r,"q_2"] & \mathsf{F}(\susp X, X') \ar[r,"{\theta \mathsf{F}(\susp f,f')}"] & \susp \mathsf{F}(Y,Y') \nospacepunct{,}
\end{tikzcd}
\end{equation}
\begin{equation} \label{verdier_rotation:ij}
\begin{tikzcd}[column sep=large]
W_1 \ar[r,"i_1"] \ar[d,"e" swap,"\cong"] & \mathsf{F}(Y,Z') \ar[r,"{\mathsf{F}(g,h')}"] \ar[d,"="] & \mathsf{F}(Z,\susp X') \ar[r,"p_1"] \ar[d,"\zeta" swap,"\cong"] & \susp W_1 \ar[d,"\susp e" swap,"\cong"] \\
W_2 \ar[r,"q'_2"] & \mathsf{F}(Y,Z') \ar[r,"{\zeta\mathsf{F}(g,h')}"] & \susp \mathsf{F}(Z,X') \ar[r,"{-\susp j'_2}"] & \susp W_2
\end{tikzcd}
\end{equation}
and analogously, there is a third isomorphism of exact triangles.
\end{lemma}
\begin{proof}
Using the pasting lemma we can write $W_1$ and $W_2$ as pushouts over $\mathsf{F}(X,X')$. Then we obtain the zig-zag of weak equivalences
\begin{equation*}
\begin{tikzcd}[row sep=small]
W_1 \ar[r,phantom,"="] &[-2em] \mathsf{F}(X,CX') \cup_{\mathsf{F}(X,X')} \mathsf{F}(Y,Y')\ar[d,"{\sim}" ] \\
& \mathsf{F}(CX,CX') \cup_{\mathsf{F}(X,X')} \mathsf{F}(Y,Y') \\
W_2 \ar[r,phantom,"="] & \mathsf{F}(CX,X') \cup_{\mathsf{F}(X,X')} \mathsf{F}(Y,Y') \ar[u,"{\sim}" swap]  
\end{tikzcd}
\end{equation*}
and we let $e\colon W_1\to W_2$ be the corresponding isomorphism in the homotopy category.
By a diagram chase one obtains commutativity of \cref{verdier_rotation:j}. 

For \cref{verdier_rotation:ij} we consider the commutative diagram
\begin{equation*}
\begin{tikzcd}
W_1 \ar[r,phantom,"="] &[-2em] \mathsf{F}(X,CX') \cup_{\mathsf{F}(X,X')} \mathsf{F}(Y,Y') \ar[r,tail] \ar[d,"\sim"] & \mathsf{F}(Y,CX') \cup_{\mathsf{F}(Y,X')} \mathsf{F}(Y,Y') \ar[d,"\sim"] \\
 W \ar[r,phantom,"\coloneqq"] & \mathsf{F}(CX,CX') \cup_{\mathsf{F}(X,X')} \mathsf{F}(Y,Y') \ar[r,tail] & \mathsf{F}(CX \cup_X Y,CX') \cup_{\mathsf{F}(Y,X')} \mathsf{F}(Y,Y') \\
W_2 \ar[r,phantom,"="] & \mathsf{F}(CX,X') \cup_{\mathsf{F}(X,X')} \mathsf{F}(Y,Y') \ar[r,tail] \ar[u,"\sim" swap] & \mathsf{F}(CX \cup_X Y,CX') \cup_{\mathsf{F}(Y,X')} \mathsf{F}(Y,Y')\ar[u,"=" swap] \nospacepunct{.} \\
\end{tikzcd}
\end{equation*}
Note that the horizontal arrow in the middle is a cofibration using that $\mathsf{F}$ is biexact and \cite[Lemma~1.4.1(1)(a)]{RadulescuBanu:2009}. The morphism $W_2\to W$ is a pushout of the cofibration $\mathsf{F}(CX, X')\to \mathsf{F}(CX, CX')$ and thus a cofibration as well. Hence their composite, the bottom horizontal arrow, is a cofibration. We obtain a commutative diagram of connecting morphisms
\begin{equation*}
\begin{tikzcd}
\mathsf{F}(Y/X, CX'/X') \ar[r,"\delta"]\ar[d,"="] & \susp W_1\ar[d] \\
\mathsf{F}(Y/X, CX'/X') \ar[r,"\delta"] & \susp W \\
\mathsf{F}(CX \cup_X Y, CX'/X') \ar[r, "\delta"] \ar[u] & \susp W_2\ar[u]
\end{tikzcd}
\end{equation*}
in $\Ho(\cat{E})$ and isomorphisms of three exact triangles. Since $\mathsf{F}(CX\cup_X Y, CX')$ is a cone for $\mathsf{F}(CX\cup_X Y, X')$ we can use \cref{rotation_stable_cof_cat} to obtain an isomorphism from the bottom exact triangle to the exact triangle
\begin{equation*}
\begin{tikzcd}
W_2\ar[r] & \mathsf{F}(Y, Y'/X') \ar[r] &\susp \mathsf{F}(CX\cup_X Y, X')\ar[r] & \susp W_2 ,
\end{tikzcd}
\end{equation*}
in which the connecting morphism is $-\susp (\mathsf{F}(CX\cup_X Y, X') \to W_2)$ and the isomorphism on the third object is $\zeta$. Finally, using the weak equivalences $CX \cup_X Y\to Z$ and $CX'\cup_X' Y'\to Z'$ we obtain the commutative diagram \cref{verdier_rotation:ij}.
\end{proof}

\subsection{Quillen bifunctors}

Recall that a bifunctor $\mathsf{F}\colon \cat{L}\times \cat{M}\to\cat{N}$ between model categories is a \emph{Quillen bifunctor} if it has a right adjoint in each variable and it satisfies the pushout-product axiom: For any cofibrations $f\colon X\to Y$ in $\cat{L}$ and $f'\colon X'\to Y'$ in $\cat{M}$, the induced morphism
\begin{equation*}
\mathsf{F}(X,Y')\cup_{\mathsf{F}(X,X')}\mathsf{F}(Y,X')\to \mathsf{F}(Y,Y')
\end{equation*}
is a cofibration which is an acyclic cofibration if $f$ or $f'$ is an acyclic cofibration. It follows that $\mathsf{F}$ has a total left derived functor $\mathrm{L}\mathsf{F}\colon \Ho(\cat{L})\times \Ho(\cat{M})\to \Ho(\cat{N})$; see \cite[Proposition~4.3.1]{Hovey:1999}. For cofibrant objects $X$ of $\cat{L}$ and $Y$ of $\cat{M}$ it can be computed by $\mathrm{L}\mathsf{F}(X,Y) = \mathsf{F}(X,Y)$.

\begin{remark}\label{cyclicshifts}
The condition that $\mathsf{F}\colon \cat{L}\times \cat{M}\to\cat{N}$ has a right adjoint in each variable implies, that the right adjoints uniquely extend to bifunctors $\operatorname{hom}^\ell \colon \op{\cat{L}}\times \cat{N}\to \cat{M}$ and $\operatorname{hom}^r \colon \op{\cat{M}}\times\cat{N}\to \cat{L}$ such that the isomorphisms
\[
\cat{L}(X,\homr{Y}{Z}) \cong \cat{N}(\mathsf{F}(X,Y),Z) \cong \cat{M}(Y,\homl{X}{Z}) .
\]
are natural in $X$, $Y$ and $Z$; see \cite[IV.7,Theorem~3]{MacLane:1998}. This is called an \emph{adjunction of two variables}; see \cite[Definition~4.1.12]{Hovey:1999}. 

If $\mathsf{F}$ is a Quillen bifunctor, then so are the cyclic shifts
\begin{alignat*}{2}
\op{\cat{N}}\times \cat{L} &\to \op{\cat{M}}  ,\quad& (Z,X) &\mapsto \homl{X}{Z} , \quad \text{and} \\
\cat{M}\times\op{\cat{N}} &\to \op{\cat{L}}  ,& (Y,Z) &\mapsto \homr{Y}{Z} ;
\end{alignat*}
see \cite[Lemma~4.2.2]{Hovey:1999}. Moreover, the adjunction of two variables induces an adjunction of two variables $(\mathrm{L}\mathsf{F},\operatorname{Rhom}^\ell,\operatorname{Rhom}^r)$ on homotopy categories; see \cite[Proposition~4.3.1]{Hovey:1999}.
\end{remark}

\begin{lemma} \label{Quillenbifunctorrestricts}
Let $\mathsf{F}\colon \cat{L}\times \cat{M}\to\cat{N}$ be a Quillen bifunctor between model categories. Then $\mathsf{F}$ restricts to a biexact functor of the corresponding categories of cofibrant objects $\mathsf{F}\colon \cat{L}_c\times \cat{M}_c\to\cat{N}_c$.
\end{lemma}
\begin{proof} For any cofibrant object $X$ in $\cat{L}$, the functor $\mathsf{F}(X,-)\colon \cat{M}\to \cat{N}$ is a left Quillen functor and thus restricts to an exact functor of cofibration categories $\mathsf{F}(X,-)\colon \cat{M}_c\to \cat{N}_c$. Similarly, $\mathsf{F}(-,Y)\colon \cat{L}_c\to \cat{N}_c$ is an exact functor of cofibration categories for any cofibrant object $Y$ of $\cat{M}$. It now follows from the pushout-product axiom that $\mathsf{F}\colon \cat{L}_c\times \cat{M}_c\to \cat{N}_c$ is a biexact functor.
\end{proof}

Recall that for any pointed model category $\cat{M}$ the homotopy category is equipped with a suspension functor $\susp\colon \Ho(\cat{M})\to \Ho(\cat{M})$. Moreover, $\cat{M}$ is called \emph{stable} if $\susp$ is an equivalence. This is equivalent to the condition that the category of cofibrant objects $\cat{M}_c$ is stable, as the suspensions commute with the equivalence $\Ho(\cat{M}_c)\to \Ho(\cat{M})$. If $\cat{M}$ is stable, then we use this equivalence to equip $\Ho(\cat{M})$ with the triangulated structure coming from $\Ho(\cat{M}_c)$. This agrees with the triangulated structure on $\Ho(\cat{M})$ from \cite[\S7]{Hovey:1999}; see the proof of \cite[Proposition~6.3.4]{Hovey:1999} to verify that the connecting morphisms coincide.

The triangulated structure on $\Ho(\cat{M})$ induces a triangulated structure on its opposite category $\op{\Ho(\cat{M})}$. Alternatively, we can equip $\op{\Ho(\cat{M})}\cong \Ho(\op{\cat{M}})$ with the triangulated structure coming from $\Ho((\op{\cat{M}})_c)$. The cofibrant objects in $\op{\cat{M}}$ are the fibrant objects in $\cat{M}$. These two triangulated structures on $\op{\Ho(\cat{M})}$ agree by \cite[Theorem~7.1.11]{Hovey:1999} with our convention for the triangulated structure on the opposite of a triangulated category.

\begin{corollary}\label{Quillenbifunctorinducesbiexact}
Let $\mathsf{F}\colon \cat{L}\times \cat{M}\to\cat{N}$ be a Quillen bifunctor between stable model categories. Then the total left derived functor $\mathrm{L}\mathsf{F}\colon \Ho(\cat{L})\times \Ho(\cat{M})\to \Ho(\cat{N})$ is biexact and admits a strong Verdier structure. Moreover, the functors $\operatorname{Rhom}^\ell$ and $\operatorname{Rhom}^r$ are biexact and admit strong Verdier structures.
\end{corollary}
\begin{proof}
This follows from \cref{cofib_biexact_verdier} and \cref{Quillenbifunctorrestricts}.
\end{proof}

\subsection{Monoidal cofibration categories} \label{monoidal_cofib_cat}

Monoidal products on triangulated categories are essential examples of biexact functors. In \cref{tt_cats} we see that many tensor triangulated categories arise as homotopy categories of stable, monoidal model categories. The unit of the monoidal product on the model category is often not a cofibrant object, so that the cofibrant replacement of the unit takes the role of the unit in the corresponding cofibration category. However, the cofibrant replacement of the unit need not satisfy the same properties as the unit. This motivates the following definition.

\begin{discussion}
A \emph{(weakly unital) monoidal cofibration category} $(\cat{C},\otimes,\unit)$ consists of a cofibration category $\cat{C}$, an object $\unit$ and a biexact functor $- \otimes - \colon \cat{C} \times \cat{C} \to \cat{C}$ with a natural isomorphism
\begin{equation*}
\alpha \colon X \otimes (Y \otimes Z) \to (X \otimes Y) \otimes Z
\end{equation*}
and natural weak equivalences
\begin{equation*}
\lambda \colon \unit \otimes X \to X \quad \text{and} \quad \rho \colon X \otimes \unit \to X
\end{equation*}
satisfying the following coherence axioms: The two canonical ways to compose the natural transformations along
\begin{equation} \label{monoidal_cofib_coherence}
\begin{gathered}
((W \otimes X) \otimes Y) \otimes Z \to W \otimes (X \otimes (Y \otimes Z))  ,\quad (\unit \otimes Y) \otimes Z \to Y \otimes Z  , \\
(X \otimes \unit) \otimes Z \to X \otimes Z, \quad \text{and} \quad (X \otimes Y) \otimes \unit \to X \otimes Y
\end{gathered}
\end{equation}
are equal.

A monoidal cofibration category is \emph{symmetric}, if it is equipped with a natural isomorphism
\begin{equation*}
\sigma \colon X \otimes Y \to Y \otimes X
\end{equation*}
satisfying the following coherence axioms: The two canonical ways to compose the natural transformations along
\begin{equation*}
X \otimes \unit \to X \quad \text{and}\quad (X \otimes Y) \otimes Z \to Z \otimes (X \otimes Y)
\end{equation*}
are equal and $\sigma^2 = \id$. 

A \emph{(left) action} of a monoidal cofibration category $(\cat{C},\otimes,\unit)$ on a cofibration category $\cat{D}$ consists of a biexact functor $\mathsf{F} \colon \cat{C} \times \cat{D}\to \cat{D}$ with a natural isomorphism
\begin{equation*}
\mathsf{F}(X,\mathsf{F}(Y,Z)) \to \mathsf{F}(X \otimes Y,Z)
\end{equation*}
and a natural weak equivalence
\begin{equation*}
\mathsf{F}(\unit, X)\to X
\end{equation*}
satisfying coherence axioms analogous to a monoidal cofibration category \cref{monoidal_cofib_coherence}, except for the last one.
\end{discussion}

\begin{remark}
A monoidal cofibration category $(\cat{C},\otimes,\unit)$ is in general \emph{not} a monoidal category since we require $\lambda$ and $\rho$ to be weak equivalences instead of isomorphisms. As a consequence, the axioms \cref{monoidal_cofib_coherence} are not superfluous, in contrast to the case of a monoidal category; cf.\@ \cite{Kelly:1964}. Moreover, when $\lambda$ and $\rho$ are isomorphisms, the coherence axioms \cref{monoidal_cofib_coherence} imply further coherence. The same need not hold when $\lambda$ and $\rho$ are weak equivalences. Explicitly, the identities $\lambda(\unit) = \rho(\unit)$ or $\sigma(\unit,\unit) = \id_{\unit \otimes \unit}$ need not be satisfied in a (symmetric) monoidal cofibration category.
\end{remark}

\begin{example}
A (symmetric) monoidal model category $(\cat{M},\otimes,\unit)$ consists of a model category $\cat{M}$ and a (symmetric) monoidal structure $(\otimes,\unit)$ on $\cat{M}$ such that $\otimes$ is a Quillen bifunctor and the \emph{unit axiom} holds: For any cofibrant replacement $Q \unit \to \unit$ and any cofibrant object $X$ the morphisms
\begin{equation*}
Q \unit \otimes X \to \unit \otimes X \quad \text{and} \quad X \otimes Q \unit \to X \otimes \unit
\end{equation*}
are weak equivalences.

If $\cat{M}$ is a (symmetric) monoidal model category, then its category of cofibrant objects $\cat{M}_c$ together with the restriction of the tensor product and a cofibrant replacement of the unit of $\cat{M}$ is a (symmetric) monoidal cofibration category.

Let $\cat{M}$ be a monoidal model category and $\cat{N}$ a (left) $\cat{M}$-model category; see \cite[Definition~4.2.18]{Hovey:1999}. Then the restriction of the action $\cat{M}\times \cat{N}\to \cat{N}$ to the cofibrant objects together with a cofibrant replacement of the tensor unit of $\cat{M}$ defines an action of the monoidal cofibration category $\cat{M}_c$ on the cofibration category $\cat{N}_c$.
\end{example}

\begin{theorem}\label{htpy_category_of_monoidal_cofcat}
Let $(\cat{C},\otimes,\unit)$ be a (symmetric) monoidal stable cofibration category. Then $(\Ho(\cat{C}),\Ho(\otimes),\gamma(\unit))$ is a (symmetric) monoidal triangulated category, and the induced monoidal product admits a strong Verdier structure.
\end{theorem}
\begin{proof}
By \cref{cofib_biexact_verdier}, the functor $\Ho(\otimes)$ is biexact and admits a strong Verdier structure. Using \cref{exact_nat_susp}, the natural isomorphism $\alpha$, the natural weak equivalences $\lambda$, $\rho$, and if available the symmetry $\sigma$ induce natural isomorphisms of exact functors in each variable on the homotopy category. These provide the (symmetric) monoidal structure for $(\Ho(\cat{C}),\Ho(\otimes),\gamma(\unit))$.
\end{proof}

\begin{discussion}\label{symmetricmonoidalclosed}
Let $(\cat{T},\otimes,\unit)$ by a tensor triangulated category. We assume $X \otimes -$ has a right adjoint $\hom{X}{-}$ for any $X \in \cat{T}$. Then $\operatorname{hom} \colon \op{\cat{T}} \times \cat{T} \to \cat{T}$ is a bifunctor, called \emph{internal hom functor}. There are induced natural isomorphisms $\theta$ and $\zeta$ as in \cref{biexact} satisfying \cref{biexact:anticommute} and $\operatorname{hom}$ is exact in the second variable. We say $(\cat{T},\otimes,\unit)$ is \emph{closed}, if $\operatorname{hom}$ is also exact in the first variable; see \cite[Definition~A.2.1]{Hovey/Palmieri/Strickland:1997}.
\end{discussion}

The homotopy category of a symmetric monoidal stable model category is a closed tensor triangulated category by \cite[Theorem~6.6.4]{Hovey:1999} together with Cisinski's proof of \cite[Conjecture~5.7.5]{Hovey:1999} in \cite[Corollaire~6.8]{Cisinski:2008}. We provide a new proof that does not rely on \cite[Conjecture~5.7.5]{Hovey:1999}. In addition, we establish strong Verdier structures for the monoidal product and the internal hom on $\Ho(\cat{M})$.

\begin{theorem}\label{htpy_category_of_monoidal_modelcat}
Let $(\cat{M},\otimes,\unit)$ be a symmetric monoidal stable model category. Then $(\Ho(\cat{M}),\Ho(\otimes),\gamma(\unit))$ is a closed symmetric monoidal triangulated category, and the induced monoidal product and internal hom functor admit a strong Verdier structure.
\end{theorem}
\begin{proof}
By \cref{htpy_category_of_monoidal_cofcat}, $(\Ho(\cat{M}),\Ho(\otimes),\gamma(\unit))$ is a tensor triangulated category and $\Ho(\otimes)$ admits a strong Verdier structure. Since $\otimes$ is a Quillen bifunctor, the corresponding internal hom functor on $\cat{M}$ induces a biexact functor admitting a strong Verdier structure on $\Ho(\cat{M})$ by \cref{Quillenbifunctorinducesbiexact}. Hence the tensor triangulated category $(\Ho(\cat{M}),\Ho(\otimes),\gamma(\unit))$ is closed.
\end{proof}

The existence of a strong Verdier structure for $\Ho(\otimes)$ is known by the analogous result for monoidal stable derivators \cite[Theorem 6.2, Lemma~6.8]{Groth/Ponto/Shulman:2014} as any symmetric monoidal stable model category $\cat{M}$ has an associated monoidal stable derivator. If in addition $\cat{M}$ is cofibrantly generated, then the associated monoidal derivator is closed; see \cite[Theorem~9.13]{Groth/Ponto/Shulman:2014}. In this case, the internal hom functor admits a strong Verdier structure by the analogous result for derivators \cite[Proposition~4.1.9]{Jin/Yang:2021}.

\section{Examples} \label{sec:examples}

In the following we discuss examples of bifunctors that admit a strong Verdier structure so that \cref{biexact_level} applies. Many of these functors are monoidal products of a tensor triangulated category or actions of a tensor triangulated category. Any action of a tensor triangulated category induces elements in the center so that \cref{kos_obj_level} holds for the corresponding Koszul objects.

\subsection{Tensor triangulated categories} \label{tt_cats}

Recall from \cref{monoidal_triangulated} and \cref{symmetricmonoidalclosed} that a closed tensor triangulated category $(\cat{T},\otimes,\unit)$ consists of a triangulated category $\cat{T}$ and a compatible closed symmetric monoidal product $\otimes$ with unit $\unit$. In particular, the monoidal product $\otimes$ and the internal hom are biexact functors in the sense of \cref{biexact}. 

\begin{discussion}\label{examples_ttc}
Each of the following tensor triangulated categories is the homotopy category of a symmetric monoidal stable model category. In particular, \cref{biexact_level} holds for the monoidal product and the internal hom functor, and \cref{kos_obj_level} for elements induced by endomorphisms of the unit. 
\begin{enumerate}[leftmargin=*]
\item The derived category $\dcat{\Mod{R}}$ of modules over a commutative ring $R$ with monoidal product $\lotimes_R$ and unit $R$. The endomorphism ring is $\End[*]{}{R} = R$. A corresponding monoidal model category is the category $\Ch{R}$ of unbounded chain complexes over $R$ with the projective model structure; see \cite[Proposition 4.2.13]{Hovey:1999}.

\item The derived category $\dcat{\Mod{RG}}$ for a commutative ring $R$ and a finite group $G$ with monoidal product $\lotimes_R$ and unit $R$. The endomorphism ring is $\End[*]{}{R} = \operatorname{H}^*(G;R)$, the group cohomology ring. A corresponding model category is the category of unbounded chain complexes $\Ch{RG}$ with the projective model structure. The monoidal product on $\Ch{RG}$ is $\otimes_R$ with the diagonal $G$-action. There is a natural isomorphism
\[
\Ch{RG}(X\otimes_R Y, Z)\cong \Ch{R}(X,\hom[R]{Y}{Z}) ,
\]
where $\hom[R]{Y}{Z}$ is the hom-complex over $R$ equipped with the conjugation action. The forgetful functor $\Ch{RG}\to \Ch{R}$ is left Quillen and in particular preserves cofibrations. By the definition of the projective model structure, a morphism in $\Ch{RG}$ is a fibration or weak equivalence if and only if it is so in $\Ch{R}$. Since $\hom[R]{-}{-}$ on $\Ch{R}$ is a Quillen bifunctor it now follows that $\hom[R]{-}{-}$ on $\Ch{RG}$ is a Quillen bifunctor and hence so is $\otimes_R$. The unit axiom can be deduced from the unit axiom for $\Ch{R}$ as well.

\item The stable module category $\Stmod{H}$ of a finite dimensional cocommutative Hopf algebra $H$ over a field $k$ with monoidal product $\otimes_k$ and unit $k$; see \cite[Proposition~4.2.15]{Hovey:1999}. This includes in particular the stable module category $\Stmod{kG}$ of a group algebra $kG$ of finite group $G$. For a group algebra, the endomorphism ring is $\End[*]{}{k} = \hat{\operatorname{H}}^*(G,k)$, the Tate cohomology ring. 

\item The homotopy category of spectra in stable homotopy theory; see for example \cite{Mandell/May/Schwede/Shipley:2001}. 

\item The homotopy category of equivariant spectra; see \cite{Mandell/May:2002}.

\item The motivic stable category in $\mathbb{A}^1$-homotopy theory; see \cite{Jardine:2000}.
\end{enumerate}
\end{discussion}

\begin{discussion}
We obtain many examples for actions of tensor triangulated categories that admit a strong Verdier structure by restricting the monoidal structures from \cref{examples_ttc} to suitable subcategories.
\begin{enumerate}[leftmargin=*]
\item Let $R$ be a commutative noetherian ring. Then $\Perf{R}$, the full subcategory of bounded complexes of finitely generated projective modules of $\dcat{\Mod{R}}$, is a tensor triangulated category with the restricted monoidal product. Further, it acts on $\dbcat{\mod{R}}$, the bounded derived category of finitely generated modules. 

\item Let $R$ be a commutative noetherian ring and $G$ a finite group. Then $\Perf{RG}$ inherits the monoidal structure from $\dcat{\Mod{RG}}$ and acts on $\dbcat{\mod{RG}}$; cf.\@ \cite[Section~8]{Buan/Krause/Snashall/Solberg:2020}.
\end{enumerate}
\end{discussion}

\subsection{Tensor products of bimodules} \label{bifunctor_bimodules}

Under reasonable conditions a stable model structure on a monoidal category transfers to stable model structures on categories of bimodules. We will show in \ref{tensorinducesbiexactwithstrongverdier} that tensoring over a monoid induces a biexact functor on homotopy categories that admits a strong Verdier structure. Moreover, the derived tensor product is part of an adjunction of two variables and the adjoints are biexact functors that admit a strong Verdier structure as well. 

\begin{discussion}\label{bicategoryofmonoids}
Let $(\cat{C},\otimes,\unit)$ be a biclosed monoidal category, that is $\otimes$ has a right adjoint in each variable denoted $\operatorname{hom}^\ell$ and $\operatorname{hom}^r$. We further assume that equalizers and coequalizers exist. We briefly recall basic constructions for bimodules; for more details see \cite{Barr:1996}.

For monoids $A$ and $B$ in $\cat{M}$, we write $\Bimod{A}{B}$ for the category of $A$-$B$-bimodules. If the monoidal structure on $\cat{C}$ is symmetric, then $\Bimod{A}{B}$ is isomorphic to the category of left $A\otimes \op{B}$-modules. The category of left $A$-modules $\Mod{A}$ can be identified with $\Bimod{A}{\unit}$ where $\unit$ is the trivial monoid. 

We write $X\otimes_B Y$ for the tensor product of a right $B$-module $X$ with a left $B$-module $Y$ defined by the coequalizer of $X\otimes B\otimes Y\rightrightarrows X\otimes Y$. Since $\otimes$ preserves coequalizers in each variable, we obtain a bifunctor
\[
- \otimes_B - \colon \Bimod{A}{B}\otimes\Bimod{B}{C}\to\Bimod{B}{C} ,
\]
where $C$ is a monoid in $\cat{C}$ as well.

If $X$ is an $A$-$B$-bimodule and $Z$ is an $A$-$C$-bimodule, then $\homl[A]{X}{Z}$ is a $B$-$C$-bimodule with $B$-action
\[
B\otimes\homl[A]{X}{Z}\to \homl[A]{X}{Z}
\]
induced by the adjoint transpose of
\[
X\otimes B\otimes \homl[A]{X}{Z}\to X\otimes B\otimes \homl{X}{Z}\to X\otimes \homl{X}{Z}\to Z ,
\]
and $C$-action 
\[
\homl[A]{X}{Z}\otimes C\to \homl[A]{X}{Z}
\]
induced by the adjoint transpose of
\[
X\otimes \homl[A]{X}{Z}\otimes C\to X\otimes \homl{X}{Z}\otimes C\to Z\otimes C\to Z .
\]

We obtain a bifunctor
\[
\homl[A]{-}{-}\colon \op{\Bimod{A}{B}}\times \Bimod{A}{C}\to \Bimod{B}{C} .
\]
Similarly, we have a bifunctor
\[
\homr[C]{-}{-}\colon \op{\Bimod{B}{C}}\times \Bimod{A}{C}\to \Bimod{A}{B} .
\]

The tensor product over a monoid is a left adjoint in each variable. Explicitly, there are natural isomorphisms
\begin{align*}
\Bimod{B}{C}(Y,\homl[A]{X}{Z}) &\cong \Bimod{A}{C}(X\otimes_B Y, Z) \\
&\cong \Bimod{A}{B}(X,\homr[C]{Y}{Z})
\end{align*}
for $X\in \Bimod{A}{B}$, $Y\in \Bimod{B}{C}$ and $Z\in \Bimod{A}{C}$ defined as follows. For a morphism of bimodules $Y\to \homl[A]{X}{Z}$, the adjoint transpose of
\[
Y\to \homl[A]{X}{Z}\to \homl{X}{Z}
\]
induces a morphism $X\otimes_B Y\to Z$ of $A$-$C$-bimodules. It is straightforward to check that this assignment defines a natural isomorphism. Analogously, one defines the second natural isomorphism.

In addition to the monoids $A$, $B$, and $C$, consider a monoid $D$. There are natural isomorphisms
\begin{equation*}
X \otimes_B (Y \otimes_C Z) \to (X \otimes_B Y) \otimes_C Z  , \quad A \otimes_A X \to X \quad \text{and} \quad X \otimes_B B \to X
\end{equation*}
for any bimodules $X \in \Bimod{A}{B}$, $Y \in \Bimod{B}{C}$, and $Z \in \Bimod{C}{D}$. The associativity isomorphism is induced by the associativity of $\cat{C}$ using that $\otimes$ preserves coequalizers in each variable. The unit isomorphisms are induced by the structure maps of $X$ as an $A$-module and as a $B$-module, respectively.

This equips the monoids in $\cat{C}$, bimodules, and morphisms of bimodules with the structure of a biclosed bicategory; see \cite{Barr:1996}. In particular, $\Bimod{A}{A}$ is a biclosed monoidal category with monoidal product $\otimes_A$ and unit $A$.
\end{discussion}

\begin{discussion}
Let $(\cat{M},\otimes,\unit)$ be a cofibrantly generated monoidal model category. In particular, the monoidal product $\otimes\colon \cat{M}\times\cat{M}\to \cat{M}$ is a Quillen bifunctor and thus part of an adjunction of two variables $(\otimes,\operatorname{hom}^\ell,\operatorname{hom}^r)$. 

We say the cofibrantly generated model structure on $\cat{M}$ \emph{(right) transfers} to the category of bimodules $\Bimod{A}{B}$, if $\Bimod{A}{B}$ is a cofibrantly generated model category with generating (trivial) cofibrations $A\otimes i\otimes B$ for generating (trivial) cofibrations $i$ of $\cat{M}$. By adjunction, this implies that the weak equivalences and fibrations of $\Bimod{A}{B}$ are the bimodule morphisms whose underlying morphisms in $\cat{M}$ are weak equivalences and fibrations, respectively. For sufficient conditions and examples when $\cat{M}$ is symmetric monoidal see \cite[Theorem~4.1 and Section~5]{Schwede/Shipley:2000}. 
\end{discussion}

\begin{lemma}\label{tensorisQuillenbifunctor}
Suppose the cofibrantly generated model structure on $\cat{M}$ transfers to the category $\Bimod{A}{B}$ for any monoids $A$ and $B$. Then
\[
- \otimes_B - \colon \Bimod{A}{B} \times \Bimod{B}{C} \to \Bimod{A}{C}
\]
is a Quillen bifunctor for any monoids $A$, $B$ and $C$ in $\cat{M}$.
\end{lemma}
\begin{proof} 
The bifunctor $-\otimes_B -$ is part of an adjunction of two variables by \cref{bicategoryofmonoids}. The pushout-product axiom can be checked on generating cofibrations and generating trivial cofibrations by \cite[Corollary~4.2.5]{Hovey:1999}. In our situation this holds, because for any morphisms $f$, $g$ in $\cat{M}$, the pushout-product of $A\otimes f \otimes B$ and $B\otimes g\otimes C$ is the pushout-product of $f$ and $g$ tensored by $A$ from the left and by $C$ from the right, and the functor $A\otimes - \otimes C$ is a left Quillen functor from $\cat{M}$ to $\Bimod{A}{C}$.
\end{proof}

\begin{proposition}\label{tensorinducesbiexactwithstrongverdier}
Let $(\cat{M},\otimes,\unit)$ be a monoidal stable model category. Suppose the model structure on $\cat{M}$ is cofibrantly generated and transfers to $\Bimod{A}{B}$ for any monoids $A$ and $B$ such that $\Bimod{A}{B}$ is a stable model category. Then the total derived functors
\begin{gather*}
- \lotimes_B - \colon \Ho(\Bimod{A}{B}) \times \Ho(\Bimod{B}{C}) \to \Ho(\Bimod{A}{C})  ,\\
\Rhomr[C]{-}{-} \colon \op{\Ho(\Bimod{B}{C})} \times \Ho(\Bimod{A}{C}) \to \Ho(\Bimod{A}{B}) , \\
\Rhoml[A]{-}{-} \colon \op{\Ho(\Bimod{A}{B})} \times \Ho(\Bimod{A}{C})\to \Ho(\Bimod{B}{C})
\end{gather*}
are biexact and admit a strong Verdier structure.
\end{proposition}
\begin{proof}
Since $-\otimes_B -$ is a Quillen bifunctor by \ref{tensorisQuillenbifunctor}, its total left derived functor is biexact and admits a strong Verdier structure by \cref{Quillenbifunctorinducesbiexact}. The functors $\homr[C]{-}{-}$ and $\homl[A]{-}{-}$ are up to switching arguments in the case of $\homl[A]{-}{-}$ the opposite functors of the cyclic shifts of $-\otimes_B -$; see \cref{cyclicshifts}. Since the the cyclic shifts are Quillen bifunctors, their total left derived functors are biexact and admit a strong Verdier structure, hence so do the opposites of these left derived functors, that is $\Rhomr[C]{-}{-}$ and $\Rhoml[A]{-}{-}$.
\end{proof}

\begin{example} \label{bimod_dg_algebra}
Let $\cat{M} = \Ch{R}$ be the model category of unbounded chain complexes over a commutative ring $R$ equipped with the projective model structure in which the weak equivalences are the quasi-isomorphisms and the fibrations are the degreewise epimorphisms; see \cite[Theorem~2.3.11]{Hovey:1999}. A monoid $A$ in $\Ch{R}$ is a differential graded $R$-algebra and the category of $A$-modules is the category of differential graded modules over $A$. The model structure of $\Ch{R}$ transfers to $\Mod{A}$; see \cite[Lemma~2.3]{Schwede/Shipley:2000}. Since every object of $\Ch{R}$ is fibrant, so is every object of $\Mod{A}$. To check that the model structure on $\Mod{A}$ is stable, we consider the loop space functor $\Omega$ instead of the suspension. The loop space functor is the suspension in $\op{\Mod{A}}$ and for a stable model category it yields the inverse of the suspension functor in the homotopy category. We set $D^{1} = \cone(\id_R)$, the chain complex with $R$ in degrees one and zero and the identity as differential. Let $X$ be an $A$-module. Then 
\[
\homr[R]{D^1}{X}\to \homr[R]{R}{X}\cong X
\]
is a fibration and $\homr[R]{D^1}{X}$ is weakly contractible. Its fiber $\Omega(X)$, that is the pullback along $0 \to R$, is a shift of $X$. Thus $\Omega$ is an equivalence on $\Ho(\Mod{A})$. The homotopy category of $\Mod{A}$ is the derived category $\dcat{A}$ of differential graded modules over $A$.

Let $A$, $B$, $C$ be differential graded algebras over a commutative ring $R$. The derived functors
\begin{gather*}
- \lotimes_B - \colon \dcat{A \otimes \op{B}} \times \dcat{B \otimes \op{C}} \to \dcat{A \otimes \op{C}}  ,\\
\RHom{\op{C}}{-}{-} \colon \op{(\dcat{B \otimes \op{C}})} \times \dcat{A \otimes \op{C}} \to \dcat{A \otimes \op{B}}\quad \text{and} \\
\RHom{A}{-}{-} \colon \op{(\dcat{A \otimes \op{B}})} \times \dcat{A \otimes \op{C}} \to \dcat{B \otimes \op{C}} .
\end{gather*}
are biexact and admit a strong Verdier structure.
\end{example}

\begin{example} \label{sec:hochschild}
Let $R$ be a commutative ring and $S$ an associative algebra over $R$. We denote by $\env[R]{S} \coloneqq S \lotimes_R \op{S}$ the \emph{derived enveloping algebra} of $S$ with respect to $R$. By choosing a projective differential graded algebra resolution $A \xrightarrow{\sim} S$ over $R$, we can view $A \otimes_R \op{A} \xrightarrow{\sim} \env[R]{S}$ as a differential graded algebra. 

By \cref{bimod_dg_algebra}, we have a monoidal triangulated category $\dcat{\Mod{\env[R]{S}}}$ with $\lotimes_S$ and unit $S$ acting on $\dcat{\Mod{S}}$ via $\lotimes_S$. The endomorphism ring of the unit $S$ is
\begin{equation*}
\Ext{\env[R]{S}}{S}{S} \eqqcolon \HH{S/R} ,
\end{equation*}
the Shukla cohomology of $S$ over $R$. When $R = k$ a field, then this coincides with Hochschild cohomology. In particular, \cref{kos_obj_level} holds for any sequence in $\HH{S/R}$.
\end{example}

\bibliographystyle{amsalpha}
\bibliography{references}

\end{document}